\documentclass[article,12pt]{amsart}

\usepackage{amsfonts}
\usepackage{amsmath}
\usepackage{graphicx}
\usepackage{color}
\usepackage{epsfig}
\usepackage{multirow}
\usepackage{color}

\numberwithin{equation}{section}
\theoremstyle{plain}
\newtheorem{thm}{Theorem}[section]

\newtheorem{prop}{Proposition}[section]
\newtheorem{cor}{Corollary}[section]
\newtheorem{lem}{Lemma}[section]

\newcommand{\cA}{\mathcal{A}}
\newcommand{\cB}{\mathcal{B}}
\newcommand{\cD}{\mathcal{D}}
\newcommand{\cF}{\mathcal{F}}
\newcommand{\cH}{\mathcal{H}}
\newcommand{\cN}{\mathcal{N}}

\newcommand{\cR}{\mathcal{R}}
\newcommand{\cT}{\mathcal{T}}
\newcommand{\cU}{\mathcal{U}}

\newcommand{\dC}{\mathbb{C}}
\newcommand{\dE}{\mathbb{E}}
\newcommand{\dI}{\mathbb{I}}
\newcommand{\dN}{\mathbb{N}}
\newcommand{\dP}{\mathbb{P}}
\newcommand{\dR}{\mathbb{R}}
\newcommand{\dV}{\mathbb{V}}
\newcommand{\dT}{\mathbb{T}}
\newcommand{\dZ}{\mathbb{Z}}

\newcommand{\dd}{\mathrm{d}}
\newcommand{\de}{\mathrm{e}}
\newcommand{\di}{\mathrm{i}}

\newcommand{\wh}{\widehat}
\newcommand{\wht}{\wh{\theta}}

\newcommand{\whr}{\wh{\rho}}

\newcommand{\hsp}{\hspace{0.5cm}}

\newcommand{\veps}{\varepsilon}

\newcommand{\VVert}{\vert\hspace{-0.025cm} \Vert}

\newcommand{\limn}{\lim_{n\, \rightarrow\, \infty}}
\newcommand{\liml}{\overset{\cD}{\longrightarrow}}

\newcommand{\cov}{\dC\text{ov}}
\renewcommand{\vec}{\textnormal{vec}}

\newcommand{\cvgas}{\hsp \textnormal{a.s.}}

\newcommand{\iid}{\overset{\text{iid}}{\sim}}

\renewcommand{\qedsymbol}{\scriptsize{$\Box$}\normalsize}
\newcommand{\udots}{\mathinner{\mskip1mu\raise1pt\vbox{\kern7pt\hbox{.}}
  \mskip2mu\raise4pt\hbox{.}\mskip2mu\raise7pt\hbox{.}\mskip1mu}}

\email{Frederic.Proia@univ-angers.fr}

\begin{document}

\title[Testing for residual correlation in the autoregressive process]
{Testing for residual correlation of any order in the autoregressive process
\vspace{2ex}}
\address{Laboratoire Angevin de REcherche en MAth\'ematiques (LAREMA), CNRS, Universit\'e d'Angers, Universit\'e Bretagne Loire. 2 Boulevard Lavoisier, 49045 Angers cedex 01.}
\author{Fr\'ed\'eric Pro\"ia}
\thanks{}

\begin{abstract}
We are interested in the implications of a linearly autocorrelated driven noise on the asymptotic behavior of the usual least squares estimator in a stable autoregressive process. We show that the least squares estimator is not consistent and we suggest a sharp analysis of its almost sure limiting value as well as its asymptotic normality. We also establish the almost sure convergence and the asymptotic normality of the estimated serial correlation parameter of the driven noise. Then, we derive a statistical procedure enabling to test for correlation of any order in the residuals of an autoregressive modelling, giving clearly better results than the commonly used portmanteau tests of Ljung-Box and Box-Pierce, and \textcolor{black}{appearing to} outperform the Breusch-Godfrey procedure on small-sized samples.
\end{abstract}

\maketitle

\vspace{-0.5cm}


\section{Introduction}


In the context of standard linear regression, the econometricians Durbin and Watson \cite{DurbinWatson50}--\cite{DurbinWatson51}--\cite{DurbinWatson71} proposed, in the middle of last century, a detailed analysis on the behavior of the estimated residual set when the driven noise \textcolor{black}{is a first-order autoregressive process}. They gave their name to a statistical procedure still commonly used nowadays, relying on a quadratic forms ratio inspired by the previous work of Von Neumann \cite{VonNeumann41}, which provides pretty good results in deciding whether the serial correlation might be considered as significative. When some of the regressors are lagged dependent random variables, and especially in the autoregressive framework, Malinvaud \cite{Malinvaud61} and shortly afterwards Nerlove and Wallis \cite{NerloveWallis66} highlighted the 	incompatibility of the Durbin-Watson procedure, potentially leading to biased and inadequate conclusions. In the 1970s, many improvements came to strengthen the study of the residual set from a least squares estimation with lagged dependent random variables, leading to statistical procedures among which we will cite the ones of Box-Pierce \cite{BoxPierce70}, of Ljung-Box \cite{LjungBox78}, of Breusch-Godfrey \cite{Breusch78}--\cite{Godfrey78} and the H-Test of Durbin \cite{Durbin70}. To the best of our knowledge, the Breusch-Godfrey procedure is actually the only one taking into account the dynamic of the generating process of the noise to test for correlation of any order in the residuals. In this paper, we discuss on two related topics that we shall now introduce.

\smallskip

The bias in the least squares estimation is a well-known issue when the driven noise is autocorrelated. As mentioned above, Malinvaud \cite{Malinvaud61} and Nerlove and Wallis \cite{NerloveWallis66} investigated the first-order autoregressive process where the driven noise is also a first-order autoregressive process, and showed that the least squares estimator is still (weakly) convergent, but remains inconsistent. Lately in 2006, Stocker \cite{Stocker06} gave substantial contributions to the study of the asymptotic bias resulting from lagged dependent regressors by considering any stationary autoregressive moving-average process as a driven noise for a dynamic modelling. As for us, we will consider the causal autoregressive process of order $p$ given, for all $t \in \dZ$, by
\begin{equation*}
\cA(L) Y_{t} = Z_{t}
\end{equation*}
where $L$ is the lag operator and $\cA$ is an autoregressive polynomial. We will also suppose that the driven noise $(Z_{t})$ itself follows the causal autoregressive process of order $q$ given, for all $t \in \dZ$, by
\begin{equation*}
\cB(L) Z_{t} = V_{t}
\end{equation*}
where $\cB$ is another autoregressive polynomial, and that the residual process $(V_{t})$ is merely a white noise. To locate the context, we start in a first part by investigating the asymptotic properties of the latter process at the heart of this paper. In a second part, we are interested in a sharp analysis of the asymptotic behavior of the least squares estimator of the parameter driving $\cA$. We establish its almost sure convergence together with its asymptotic normality and some rates of convergence. This may as well be seen as a natural extension of the recent work of Bercu and Pro\"ia \cite{BercuProia13} and of Pro\"ia \cite{Proia13} in which the almost sure convergence and the asymptotic normality are established for $q=1$ and $p=1$, and for $q=1$ and any $p \geq 1$, respectively, under some stability assumptions that we also weaken here. One can also cite \cite{BitsekiDjelloutProia13} where moderate deviations come to strengthen the convergences in the particular case where $p=1$ and $q=1$.

\smallskip

In the third part, we suggest an (of course biased) estimator of the parameter driving $\cB$, again using a least squares methodology. Its almost sure convergence is established together with its asymptotic normality under the null of absence of serial correlation in the residuals. We shall make here the parallel with the procedures of Box-Pierce \cite{BoxPierce70} and Ljung-Box \cite{LjungBox78} on the one hand, and with the procedure of Breusch-Godfrey \cite{Breusch78}--\cite{Godfrey78} on the other hand. The statistical procedure that we propose in a last part is inspired by the well-known procedure due to the eponymous econometricians, Durbin and Watson. Since their seminal work \cite{DurbinWatson50}--\cite{DurbinWatson51}--\cite{DurbinWatson71}, many improvements have been brought to the Durbin-Watson statistic, related to its behavior under the null and to the power of the associated procedure. One can cite for example the works of Maddala and Rao \cite{MaddalaRao73} in 1973, of Park \cite{Park75} in 1975, of Inder \cite{Inder84}--\cite{Inder86} in the 1980s, of Durbin \cite{Durbin86} in 1986, or of King and Wu \cite{KingWu91} in 1991. Whereas the upper and lower bounds of the test were previously established by a Monte-Carlo study, the asymptotic normality of the Durbin-Watson statistic is suggested without proof in \cite{Durbin70} under strong assumptions (such as gaussianity), and finally proved under less restrictive hypothesis in \cite{BercuProia13}--\cite{Proia13} for $p \geq 1$ and $q=1$, even under the alternative. \textcolor{black}{We conclude the paper by giving a quick summary of some simulation results about} the empirical power of our procedure, compared with the ones mentioned above. We observe in particular that it seems to give better results than the usual procedures on small-sized samples. The proofs of our results are postponed to the Appendix, they essentially rely on a martingale approach \cite{Duflo97}--\cite{HallHeyde80}.

\smallskip

\noindent \textbf{Notations.} In the whole paper, for any matrix $M$, $M^{\prime}$ is the transpose and $\VVert M \VVert$ is the spectral norm of $M$. For any square matrix $M$, $\textnormal{tr}(M)$, $\det(M)$, $\lambda_{\text{min}}(M)$ and $\lambda_{\text{max}}(M)$ are the trace, the determinant, the smallest eigenvalue and the largest eigenvalue of $M$, respectively. For any vector $v$, $\Vert v \Vert$ stands for the euclidean norm, $\Vert v \Vert_1$ and $\Vert v \Vert_{\infty}$ for the 1--norm and the infinite norm of $v$. The identity and the exchange matrices of any order $h \in \dN^{*}$ will be respectively called
\begin{equation*}
I_{h} = \begin{pmatrix}
1 & 0 & \hdots & 0\\
0 & 1 & \hdots & 0\\
\vdots & \vdots & \ddots & \vdots\\
0 & 0 & \hdots & 1
\end{pmatrix} \hsp \text{and} \hsp J_{h} = \begin{pmatrix}
0 & \hdots & 0 & 1\\
0 & \hdots & 1 & 0\\
\vdots & \udots & \vdots & \vdots\\
1 & \hdots & 0 & 0
\end{pmatrix}.
\end{equation*}
Finally, $M \otimes N$ is the Kronecker product between any matrices $M$ and $N$, and $\vec(M)$ is the vectorialization of $M$.


\section{On the asymptotic properties of the process}


For all $t \in \dZ$, we consider the process given by
\begin{equation}
\label{ModARGen}
\left\{
\begin{array}{c}
Y_{t} = \theta^{\, \prime} \Phi_{t-1}^{p} + Z_{t} \vspace{1ex}\\
Z_{t} = \rho^{\, \prime} \Psi_{t-1}^{q} + V_{t} 
\end{array}
\right.
\end{equation}
where
\begin{equation}
\label{PhiPsi}
\Phi_{t}^{p} = \begin{pmatrix}
Y_{t} & Y_{t-1} & \hdots & Y_{t-p+1}
\end{pmatrix}^{\prime} \hsp \text{and} \hsp \Psi_{t}^{q} = \begin{pmatrix}
Z_{t} & Z_{t-1} & \hdots & Z_{t-q+1}
\end{pmatrix}^{\prime}
\end{equation}
for the couple of parameters $p \in \dN$ and $q \in \dN$. The residual process $(V_{t})$ is a white noise such that $\dE[V_1^2] = \sigma^2 > 0$ and $\dE[V_1^4] = \tau^4 < \infty$. We assume that the polynomials
\begin{equation*}
\cA(z) = 1 - \theta_1\, z - \hdots - \theta_{p}\, z^{p} \hsp \text{and} \hsp \cB(z) = 1 - \rho_1\, z - \hdots - \rho_{q}\, z^{q}
\end{equation*}
are causal, that is $\cA(z) \neq 0$ and $\cB(z) \neq 0$ for all $z \in \dC$ such that $\vert z \vert \leq 1$. We also assume that $\theta_{p} \neq 0$, hence that $(Y_{t})$ is at least an AR($p$) process. Let us start by a little but useful technical lemma.

\begin{lem}
\label{LemCausal}
Under the causality assumptions on $\cA$ and $\cB$, $(Y_{t})$ is a causal and ergodic AR($p+q$) process defined, for all $t \in \dZ$, as
\begin{equation}
\label{ModAR}
Y_{t} = \beta^{\, \prime} \Phi_{t-1}^{p+q} + V_{t}
\end{equation}
where, for all $1 \leq k \leq p+q$,
\begin{equation}
\label{Beta}
\beta_{k} = \rho_{k} - \theta_1\, \rho_{k-1} - \theta_2\, \rho_{k-2} - \hdots - \theta_{k-1}\, \rho_1 + \theta_{k}
\end{equation}
with the convention that $\theta_{j} = 0$ for $j \notin \{ 1, \hdots, p \}$ and $\rho_{j} = 0$ for $j \notin \{ 1, \hdots, q \}$.
\end{lem}
\begin{proof}
See Appendix.
\end{proof}

\noindent From the causality of the autoregressive polynomial, we deduce that $(Y_{t})$ is stationary. Accordingly, let
\begin{equation}
\label{StatL0}
\ell_0 = \dV(Y_1) > 0
\end{equation}
be the variance of the process, positive as soon as $\dE[V_1^2] = \sigma^2 > 0$, and, for all $h \in \dN^{*}$, let
\begin{equation}
\label{StatLh}
\ell_{h} = \cov(Y_1, Y_{1+h})
\end{equation}
be the successive covariances. Denote by $\Delta_{h}$ the associated Toeplitz covariance matrix of order $h$, that is
\begin{equation}
\label{Delta}
\Delta_{h} = \begin{pmatrix}
\ell_0 & \ell_1 & \hdots & \ell_{h-1} \\
\ell_1 & \ddots & \ddots & \vdots \\
\vdots & \ddots & \ddots & \ell_1 \\
\ell_{h-1} & \hdots & \ell_1 & \ell_0
\end{pmatrix}.
\end{equation}

\begin{lem}
\label{LemInvD}
Under the causality assumptions on $\cA$ and $\cB$, the Toeplitz matrix $\Delta_{h}$ is positive definite, for all $h \geq 1$.
\end{lem}
\begin{proof}
See Appendix.
\end{proof}

\noindent The invertibility of $\Delta_{h}$ implies that, in the particular case where $h=p+q+1$, the Yule-Walker equations have a unique solution giving the expression of $\beta$ and $\sigma^2$ according to $\ell_0, \hdots, \ell_{p+q}$. 	Conversely, one can see that the linear system
\begin{equation*}
B\, \Lambda_0^{p+q+1} = U
\end{equation*}
where the square matrix $B$ of order $p+q+1$ is given by
\begin{equation*}
B =
\begin{pmatrix}
1 & -\beta_1 & -\beta_2 & \hdots & \hdots & -\beta_{p+q-2} & -\beta_{p+q-1} & -\beta_{p+q} \\
-\beta_1 & 1-\beta_2 & -\beta_3 & \hdots & \hdots & -\beta_{p+q-1} & -\beta_{p+q} & 0 \\
-\beta_2 & -\beta_1-\beta_3 & 1-\beta_4 & \hdots & \hdots & -\beta_{p+q} & 0 & 0\\
\vdots & \vdots & \vdots & & & \vdots & \vdots & \vdots\\
\vdots & \vdots & \vdots & & & \vdots & \vdots & \vdots\\
-\beta_{p+q-1} & -\beta_{p+q-2}-\beta_{p+q} & -\beta_{p+q-3} & \hdots & \hdots & -\beta_1 & 1 & 0\\
-\beta_{p+q} & -\beta_{p+q-1} & -\beta_{p+q-2} & \hdots & \hdots & -\beta_2 & -\beta_1 & 1\\
\end{pmatrix}
\end{equation*}
and the vectors $\Lambda_0^{p+q+1}$ and $U$ of order $p+q+1$ are defined as
\begin{equation*}
\Lambda_0^{p+q+1} = \begin{pmatrix}
\ell_0 & \ell_1 & \hdots & \ell_{p+q}
\end{pmatrix}^{\prime} \hsp \text{and} \hsp U = \begin{pmatrix}
\sigma^2 & 0 & \hdots & 0
\end{pmatrix}^{\prime},
\end{equation*}
has the unique solution given by
\begin{equation}
\label{VecCov}
\Lambda_0^{p+q+1} = B^{-1} U.
\end{equation}
This enables to express $\ell_0, \hdots, \ell_{p+q}$ in terms of $\beta$ and $\sigma^2$. Let us conclude this part by introducing some more notations that will be useful thereafter. In all the sequel, we decompose $\beta$ into
\begin{equation}
\label{AlphaGamma}
\alpha = \begin{pmatrix} \beta_1 & \hdots & \beta_{p} \end{pmatrix}^{\prime} \hsp \text{and} \hsp \gamma = \begin{pmatrix} \beta_{p+1} & \hdots & \beta_{p+q} \end{pmatrix}^{\prime}.
\end{equation}
The Toeplitz matrix $C_{\alpha}$ and the Hankel matrix $C_{\gamma}$ of order $q \times q$ are given by
\begin{equation}
\label{CaCg}
C_{\alpha} = \begin{pmatrix}
0 & \hdots & \hdots & \hdots & 0 \vspace{-0.1cm} \\
\alpha_1 & \ddots & & & \vdots \vspace{-0.1cm} \\
\alpha_2 & \alpha_1 & \ddots & & \vdots \\
\vdots & \ddots & \ddots & \ddots & \vdots \\
\alpha_{q-1} & \hdots & \alpha_2 & \alpha_1 & 0
\end{pmatrix} \hsp \text{and} \hsp C_{\gamma} = \begin{pmatrix}
\gamma_1 & \gamma_2 & \hdots & \hdots & \gamma_{q} \\
\gamma_2 & & & \gamma_{q} & 0 \\
\vdots & & \udots & \udots & \vdots \\
\vdots & \udots & \udots & & \vdots \\
\gamma_{q} & 0 & \hdots & \hdots & 0
\end{pmatrix}.
\end{equation}
Note that $C_{\alpha}$ is only well-defined for the usual case $p \geq q$ and is generated by $(0, \alpha_1, \hdots, \alpha_{q-1})$. If $p < q$, it shall be replaced by the Toeplitz matrix $G_{\alpha}$ of order $q \times q$ generated by $(0, \alpha_1, \hdots, \alpha_{p}, \gamma_1, \hdots, \gamma_{q-p-1})$. Likewise, the Hankel matrix $D$ of order $p \times q$ is given by
\begin{equation}
\label{D}
D = \begin{pmatrix}
\alpha_1 & \alpha_2 & \hdots & \hdots & \alpha_{q} \\
\alpha_2 & & & \udots & \vdots \\
\vdots & & \udots & & \alpha_{p} \vspace{-0.1cm} \\
\vdots & \udots & & \udots & \gamma_1 \vspace{0.1cm} \\
\alpha_{q} & & \udots & \udots & \gamma_2 \\
\vdots & \udots & \udots & \udots & \vdots \\
\alpha_{p} & \gamma_1 & \gamma_2 & \hdots & \gamma_{q-1} \\
\end{pmatrix}
\end{equation}
and is well-defined only for $p \geq q$, whereas it shall be replaced for $p < q$ by the Hankel matrix $E$ of order $p \times q$ given by
\begin{equation}
\label{E}
E = \begin{pmatrix}
\alpha_1 & \hdots & \hdots & \alpha_{p} & \gamma_1 & \hdots & \gamma_{q-p} \\
\vdots & & \udots & \udots & & \udots & \vdots \\
\vdots & \udots & \udots  & & \udots & & \vdots \\
\alpha_{p} & \gamma_1 & \hdots & \gamma_{q-p} & \hdots & \hdots & \gamma_{q-1}
\end{pmatrix}.
\end{equation}


\section{On the behavior of the least squares estimator of $\theta$}
\label{SecEstT}


Consider an observed path $(Y_{t}(\omega))$ of the process \eqref{ModAR} on $\{ 1 \leq t \leq n \}$ and, to lighten all calculations, having initial values $Y_{-(p+q)}, \hdots, Y_0$ set to zero. The associated process $(Y_{n})$ is therefore totally described by the $\sigma$--algebra $\cF_{n} = \sigma(V_1, \hdots, V_{n})$. We obviously infer from the last section that $(Y_{n})$ is asymptotically stationary and that it satisfies, for all $h \in \{1, \hdots, p+q \}$,
\begin{equation*}
\limn \dV(Y_{n}) = \ell_0 > 0 \hsp \text{and} \hsp \limn \cov(Y_{n}, Y_{n+h}) = \ell_{h}
\end{equation*}
where $\ell_0, \hdots, \ell_{p+q}$ are given in \eqref{VecCov}. Now, denote by $P_{n}$ the matrix of order $p \times q$ given, for all $n \geq 1$, by
\begin{equation}
\label{Pn}
P_{n} = \sum_{t=1}^{n} \begin{pmatrix} \Phi_{t-1}^{p} Y_{t} & \Phi_{t-2}^{p} Y_{t} & \hdots & \Phi_{t-q}^{p} Y_{t}
\end{pmatrix}
\end{equation}
and by $S_{n}$ the square matrix of order $p$ given by
\begin{equation}
\label{Sn}
S_{n} = \sum_{t=0}^{n} \Phi_{t}^{p}\, \Phi_{t}^{p\,\, \prime} + S
\end{equation}
where $\Phi_{t}^{p}$ is described in \eqref{PhiPsi} and $S$ is a symmetric and positive definite matrix of order $p$ added to avoid a useless invertibility assumption on $S_{n}$. Let also
\begin{equation}
\label{Tn}
T_{n} = S_{n-1}^{-1}\, P_{n}.
\end{equation}
We shall start by studying the asymptotic behavior of $T_{n}$. As a matter of fact, we will see in the sequel that the least squares estimator of $\theta$ in \eqref{ModARGen} is closely related to $T_{n}$. Denote by $\Pi_{pq}$ the Hankel matrix of order $p \times q$ given by
\begin{equation}
\label{Pi}
\Pi_{pq} = \begin{pmatrix}
\ell_1 & \ell_2 & \hdots & \ell_{q} \\
\ell_2 & \ell_3 & \hdots & \ell_{q+1} \\
\vdots & \vdots & & \vdots \\
\ell_{p} & \ell_{p+1} & \hdots & \ell_{p+q}
\end{pmatrix}
\end{equation}
and by $K$ the matrix of order $p\,q \times p\,q$ given by
\begin{equation}
\label{K}
K = \left\{
\begin{array}[c]{ll}
(I_{q} - C_{\alpha}) \otimes I_{p}  -  C_{\gamma} \otimes J_{p} & \hsp \text{for } p \geq q\\
(I_{q} - G_{\alpha})\hspace{-0.02cm} \otimes I_{p}  -  C_{\gamma} \otimes J_{p} & \hsp \text{for } p < q
\end{array}
\right.
\end{equation}
where $C_{\alpha}$, $G_{\alpha}$ and $C_{\gamma}$ are defined in \eqref{CaCg}.

\begin{prop}
\label{PropCvgTn}
Under the causality assumptions on $\cA$ and $\cB$ and as soon as $\dE[V_1^2] = \sigma^2 < \infty$, we have the almost sure convergence
\begin{equation*}
\limn T_{n} = T^{*} \cvgas
\end{equation*}
where the limiting value is given by
\begin{equation}
\label{TLim}
T^{*} = \Delta_{p}^{-1}\, \Pi_{pq}
\end{equation}
in which $\Delta_{p}$ and $\Pi_{pq}$ are given in \eqref{Delta} and \eqref{Pi}, respectively. \textcolor{black}{Under the additional assumption that $K$ is invertible,} we have the almost sure convergence
\begin{equation*}
\limn \vec(T_{n}) = \left\{
\begin{array}[c]{ll}
K^{-1}\, \vec(D) \cvgas & \hsp \text{for } p \geq q\\
K^{-1}\: \vec(E) \cvgas & \hsp \text{for } p < q
\end{array}
\right.
\end{equation*}
where $D$ and $E$ are given in \eqref{D} and \eqref{E}.
\end{prop}
\begin{proof}
See Appendix.
\end{proof}

\noindent Of course, we have
\begin{equation*}
\vec(T^{*}) = K^{-1}\, \vec(D) \hsp \text{or} \hsp \vec(T^{*}) = K^{-1}\, \vec(E).
\end{equation*}
However, the latter expressions seem more elegant than \eqref{TLim} in the sense that they only depend on $\theta$ and $\rho$, and do not require the computation of $\ell_0, \hdots, \ell_{p+q}$. In addition, denote by $\theta^{*}$ the first column of $T^{*}$, that is
\begin{equation}
\label{EstTLim}
T^{*} = \begin{pmatrix}
\theta^{*} & * & \hdots & *
\end{pmatrix}
\end{equation}
which can also be seen as the first $p$ elements of $\vec(T^{*})$. We are now going to study the asymptotic behavior of the least squares estimator of $\theta$ in \eqref{ModARGen}, that is
\begin{equation}
\label{EstT}
\wht_{n} = S_{n-1}^{-1}\, \sum_{t=1}^{n} \Phi_{t-1}^{p}\, Y_{t}.
\end{equation}

\noindent We assume here that a statistical argument (such as the empirical autocorrelation function) has suggested an autoregression of order $p$, at least.

\begin{thm}
\label{ThmCvgT}
Under the causality assumptions on $\cA$ and $\cB$ and as soon as $\dE[V_1^2] = \sigma^2 < \infty$, we have the almost sure convergence
\begin{equation*}
\limn \wht_{n} = \theta^{*} \cvgas
\end{equation*}
where the limiting value $\theta^{*}$ is given by \eqref{EstTLim}.
\end{thm}
\begin{proof}
\textcolor{black}{It is a corollary of Proposition \ref{PropCvgTn}.}
\end{proof}

\noindent \textcolor{black}{Note that, strictly speaking, the invertibility of $K$ is not needed for the last result, in virtue of \eqref{TLim}.} Our next result is related to the asymptotic normality of $T_{n}$, and shall lead to the asymptotic normality of $\wht_{n}$. For all $h, k \in \{0, \hdots, q \}$, let
\begin{equation}
\label{Gammahk}
\Gamma_{h-k} = \begin{pmatrix}
\ell_{h-k} & \hdots & \ell_{h-k-p+1} \\
\vdots & \ddots & \vdots \\
\ell_{h-k+p-1} & \hdots & \ell_{h-k}
\end{pmatrix}
\end{equation}
and note that $\Gamma_0 = \Delta_{p}$ and that $\Gamma_{h-k} = \Gamma_{k-h}^{\, \prime}$. Hence, consider the symmetric matrix of order $p\, q \times p\, q$ given by
\begin{equation}
\label{Gamma}
\Gamma_{\!pq} = \begin{pmatrix}
\Delta_{p} & \Gamma_1^{\, \prime} & \hdots & \Gamma_{q-1}^{\, \prime} \\
\Gamma_1 & \Delta_{p} & \hdots & \Gamma_{q-2}^{\, \prime} \\
\vdots & \vdots & \ddots & \vdots \\
\Gamma_{q-1} & \Gamma_{q-2} & \hdots & \Delta_{p}
\end{pmatrix}.
\end{equation}

\begin{prop}
\label{PropTlcTn}
Under the causality assumptions on $\cA$ and $\cB$, \textcolor{black}{if we assume that $K$ is invertible} and that $\dE[V_1^4] = \tau^4 < \infty$, then we have the asymptotic normality
\begin{equation*}
\sqrt{n}\, \big( \vec(T_{n}) - \vec(T^{*}) \big) \liml \cN\big( 0 , \Sigma_{T} \big)
\end{equation*}
where the limiting covariance is given by
\begin{equation}
\label{TCov}
\Sigma_{T} = \sigma^2\, K^{-1}\, (I_{q} \otimes \Delta_{p}^{-1})\, \Gamma_{\!pq}\, (I_{q} \otimes \Delta_{p}^{-1})\, K^{\, \prime\, -1}
\end{equation}
in which $\Delta_{p}$ and $\Gamma_{\!pq}$ are given in \eqref{Delta}, \eqref{Gamma}, and $K$ is given in \eqref{K}.
\end{prop}
\begin{proof}
See Appendix.
\end{proof}

\noindent Note that $\Sigma_{T}$ does not depend on $\sigma^2$, despite appearances. Indeed, there is a factor $\sigma^{-2}$ in $\Delta_{p}^{-1}$ and a factor $\sigma^2$ in $\Gamma_{\!pq}$. In addition, denote by $\Sigma_{\theta}$ the top left-hand block matrix of order $p$ in $\Sigma_{T}$, that is
\begin{equation}
\label{EstTCov}
\Sigma_{T} = \begin{pmatrix}
\Sigma_{\theta} & * & \hdots & * \\
* & & & \vdots \\
\vdots & & & \vdots \\
* & \hdots & \hdots & *
\end{pmatrix}.
\end{equation}
\noindent It follows that we have the following result on $\wh{\theta}_{n}$.

\begin{thm}
\label{ThmTlcT}
Under the causality assumptions on $\cA$ and $\cB$, \textcolor{black}{if we assume that $K$ is invertible} and that $\dE[V_1^4] = \tau^4 < \infty$, then we have the asymptotic normality
\begin{equation*}
\sqrt{n}\, \big( \wh{\theta}_{n} - \theta^{*} \big) \liml \cN\big( 0 , \Sigma_{\theta} \big)
\end{equation*}
where the limiting covariance $\Sigma_{\theta}$ is given by \eqref{EstTCov}.
\end{thm}
\begin{proof}
\textcolor{black}{It is a corollary of Proposition \ref{PropTlcTn}.}
\end{proof}

\noindent Let us now establish somes rates of almost sure convergence.

\begin{prop}
\label{PropRatTn}
Under the causality assumptions on $\cA$ and $\cB$, \textcolor{black}{if we assume that $K$ is invertible} and that $\dE[V_1^4] = \tau^4 < \infty$, then we have the quadratic strong law
\begin{equation*}
\limn \frac{1}{\log n} \sum_{t=1}^{n} \big( \vec(T_{t}) - \vec(T^{*}) \big) \big( \vec(T_{t}) - \vec(T^{*}) \big)^{\prime} = \Sigma_{T} \cvgas
\end{equation*}
where the limiting value $\Sigma_{T}$ is given by \eqref{TCov}. In addition, we also have the law of iterated logarithm
\begin{equation*}
\limsup_{n\, \rightarrow\, \infty} \left( \frac{n}{2 \log \log n} \right)\, \big( \vec(T_{n}) - \vec(T^{*}) \big) \big( \vec(T_{n}) - \vec(T^{*}) \big)^{\prime} = \Sigma_{T} \cvgas
\end{equation*}
\end{prop}
\begin{proof}
See Appendix.
\end{proof}

\noindent Whence we deduce the following result on $\wh{\theta}_{n}$.

\begin{thm}
\label{ThmRatT}
Under the causality assumptions on $\cA$ and $\cB$, \textcolor{black}{if we assume that $K$ is invertible} and that $\dE[V_1^4] = \tau^4 < \infty$, then we have the quadratic strong law
\begin{equation*}
\limn \frac{1}{\log n} \sum_{t=1}^{n} \big( \wh{\theta}_{t} - \theta^{*} \big) \big( \wh{\theta}_{t} - \theta^{*} \big)^{\prime} = \Sigma_{\theta} \cvgas
\end{equation*}
where the limiting value $\Sigma_{\theta}$ is given by \eqref{EstTCov}. In addition, we also have the law of iterated logarithm
\begin{equation*}
\limsup_{n\, \rightarrow\, \infty} \left( \frac{n}{2 \log \log n} \right)\, \big( \wh{\theta}_{n} - \theta^{*} \big) \big( \wh{\theta}_{n} - \theta^{*} \big)^{\prime} = \Sigma_{\theta} \cvgas
\end{equation*}
\end{thm}
\begin{proof}
From Proposition \ref{PropRatTn}, the proof is once again immediate.
\end{proof}

\noindent One can accordingly establish the rate of almost sure convergence of the cumulative quadratic errors of $\wh{\theta}_{n}$, that is
\begin{equation}
\label{EstTRatTr}
\limn \frac{1}{\log n} \sum_{t=1}^{n} \big\Vert \wh{\theta}_{t} - \theta^{*} \big\Vert^2 = \text{tr}(\Sigma_{\theta}) \cvgas
\end{equation}
Moreover, we also have the rate of almost sure convergence
\begin{equation}
\label{EstTRatNorm}
\big\Vert \wh{\theta}_{n} - \theta^{*} \big\Vert^2 = O\!\left( \frac{\log \log n}{n} \right) \cvgas
\end{equation}

\noindent We will conclude this section by comparing our results with the ones established in \cite{Proia13}, for $q=1$ and under stronger assumptions on the parameters (namely, $\Vert \theta \Vert_1 < 1$ and $\vert \rho \vert < 1$). In our framework, $T_{n}$ and $\wh{\theta}_{n}$ coincide when $q=1$ and by extension $T^{*}$ and $\theta^{*}$ also coincide, just as $\Sigma_{T}$ and $\Sigma_{\theta}$. Then \textit{via} \eqref{Beta} and \eqref{AlphaGamma}, we obtain
\begin{equation*}
\alpha = \begin{pmatrix} \theta_1 + \rho & \theta_2 - \theta_1 \rho & \hdots & \theta_{p} - \theta_{p-1} \rho \end{pmatrix}^{\prime} \hsp \text{and} \hsp \gamma = -\theta_{p}\, \rho
\end{equation*}
which leads to $
C_{\alpha} = 0$, $C_{\gamma} = -\theta_{p} \rho$ and $D = \alpha$, using the notations of \eqref{CaCg} and \eqref{D}. Accordingly, $K$ in \eqref{K} now satisfies
\begin{equation*}
K = I_{p} + \theta_{p}\, \rho\, J_{p} \hsp \text{and} \hsp K^{-1} = \frac{I_{p} - \theta_{p}\, \rho\, J_{p}}{(1 - \theta_{p}\, \rho) (1 + \theta_{p}\, \rho)}.
\end{equation*}
Hence, Theorem 2.1 of \cite{Proia13} and Theorem \ref{ThmCvgT} are clearly equivalent for $q=1$. Similarly, one can see that $\Sigma_{T}$ in \eqref{TCov} becomes
\begin{equation*}
\Sigma_{\theta} = \frac{\sigma^2\, (I_{p} + \theta_{p}\, \rho\, J_{p}) \Delta_{p}^{-1} (I_{p} + \theta_{p}\, \rho\, J_{p})}{(1 - \theta_{p}\, \rho)^2 (1 + \theta_{p}\, \rho)^2}
\end{equation*}
since from \eqref{Gamma}, $\Gamma_{\!pq} = \Delta_{p}$. From the definition of the slightly different covariance matrix $\Delta_{p}$ used in \cite{Proia13} (see Lemma B.3), the related Theorem 2.2 coincides with our Theorem \ref{ThmTlcT}. In conclusion, this work on the least squares estimator of $\theta$ in an autoregressive process driven by another autoregressive process may be seen as a wide generalization of \cite{BercuProia13}--\cite{Proia13} since we consider that $q \geq 1$ and since our set of hypothesis is far less restrictive on our parameters. We shall now, in a next part, study the residual set generated by this biased estimation of $\theta$.


\section{On the behavior of the least squares estimator of $\rho$}
\label{SecEstR}


The first step is to construct a residual set $(\wh{Z}_{n})$ associated with the observed path $(Y_{n})$. For all $1 \leq t \leq n$, let
\begin{eqnarray}
\label{EstRes}
\wh{Z}_{t} & = & Y_{t} - \wh{\theta}_{n}^{\: \prime}\, \Phi_{t-1}^{p}\nonumber \\
 & = & Y_{t} - \wh{\theta}_{1,\, n} Y_{t-1} - \hdots - \wh{\theta}_{p,\, n} Y_{t-p}
\end{eqnarray}
where we recall that $Y_{1-p} = \hdots = Y_0 = 0$. A natural way to estimate $\rho$ in \eqref{ModARGen} using least squares is to consider the estimator given, for all $n \geq 1$, by
\begin{equation}
\label{EstR}
\whr_{n} = \wh{J}_{n-1}^{\, -1}\, \sum_{t=1}^{n} \wh{\Psi}_{t-1}^{q}\, \wh{Z}_{t}
\end{equation}
where the square matrix $\wh{J}_{n}$ of order $q$ is given by
\begin{equation}
\label{Jn}
\wh{J}_{n} = \sum_{t=0}^{n} \wh{\Psi}_{t}^{q}\, \wh{\Psi}_{t}^{q\,\, \prime} + J
\end{equation}
in which $J$ is a symmetric and positive definite matrix of order $q$ added to avoid a useless invertibility assumption on $\wh{J}_{n}$, and $\wh{\Psi}_{t}^{q}$ is defined as
\begin{equation}
\label{PsiHat}
\wh{\Psi}_{t}^{q} = \begin{pmatrix}
\wh{Z}_{t} & \wh{Z}_{t-1} & \hdots & \wh{Z}_{t-q+1}
\end{pmatrix}^{\prime}.
\end{equation}
We obviously consider that $\wh{Z}_{1-q} = \hdots = \wh{Z}_0 = 0$, to simplify the calculations. It will be easier to characterize the limiting value of $\whr_{n}$ by introducing some more notations. For all $h \in \{ -q-p, \hdots, p+q+1-k \}$ and $k \in \{ p, q \}$, let
\begin{equation}
\label{Lambda}
\Lambda_{h}^{k} = \begin{pmatrix}
\ell_{h} & \ell_{h+1} & \hdots & \ell_{h+k-1}
\end{pmatrix}^{\prime}
\end{equation}
where the asymptotic covariances are given in \eqref{VecCov} and satisfy $\ell_{h} = \ell_{-h}$, by (asymptotic) stationarity. Denote by $L_{q}$ the square matrix of order $q$ given by
\begin{equation}
\label{Lq}
L_{q} = \begin{pmatrix}
\theta^{*\, \prime} \Lambda_1^{p} & \theta^{*\, \prime} \Lambda_0^{p} & \hdots & \theta^{*\, \prime} \Lambda_{-q+2}^{p} \\
\theta^{*\, \prime} \Lambda_2^{p} & \theta^{*\, \prime} \Lambda_1^{p}  & \hdots & \theta^{*\, \prime} \Lambda_{-q+3}^{p} \\
\vdots & \vdots & \ddots & \vdots \\
\theta^{*\, \prime} \Lambda_{q}^{p} & \theta^{*\, \prime} \Lambda_{q-1}^{p} & \hdots & \theta^{*\, \prime} \Lambda_1^{p}
\end{pmatrix}.
\end{equation}
Denote also by $D_{q}$ the square matrix of order $q$ given by
\begin{equation}
\label{Dq}
D_{q} = \begin{pmatrix}
\theta^{*\, \prime} \Delta_{p}\, \theta^{*} & \theta^{*\, \prime}\, \Gamma_1^{\, \prime} \, \theta^{*} & \hdots & \theta^{*\, \prime}\, \Gamma_{q-1}^{\, \prime} \, \theta^{*} \\
\theta^{*\, \prime}\, \Gamma_1 \, \theta^{*} & \theta^{*\, \prime} \Delta_{p}\, \theta^{*} & \hdots & \theta^{*\, \prime}\, \Gamma_{q-2}^{\, \prime} \, \theta^{*} \\
\vdots & \vdots & \ddots & \vdots \\
\theta^{*\, \prime}\, \Gamma_{q-1} \, \theta^{*} & \theta^{*\, \prime}\, \Gamma_{q-2} \, \theta^{*} & \hdots & \theta^{*\, \prime} \Delta_{p}\, \theta^{*}
\end{pmatrix}
\end{equation}
where the matrices $\Gamma_{h-k}$ are defined in \eqref{Gammahk}. Finally, let
\begin{equation}
\label{Eq}
E_{q} = \begin{pmatrix}
\theta^{*\, \prime} ( \Lambda_2^{p} + \Lambda_0^{p} - \Gamma_1 \, \theta^{*}) \\
\theta^{*\, \prime} (\Lambda_3^{p} + \Lambda_{-1}^{p} - \Gamma_2 \, \theta^{*}) \\
\vdots \\
\theta^{*\, \prime} (\Lambda_{q+1}^{p} + \Lambda_{-q+1}^{p} - \Gamma_{q} \, \theta^{*}) \\
\end{pmatrix}.
\end{equation}
We are now in the position to build the limiting value $\rho^{*}$, only depending on the asymptotic covariances $\ell_0, \hdots, \ell_{p+q}$ computed from the Yule-Walker equations. Using the whole notations above,
\begin{equation}
\label{EstRLim}
\textcolor{black}{\rho^{*} = \Xi_{q}^{\, -1} \big( \Lambda_1^{q} - E_{q} \big) \hsp \text{where} \hsp \Xi_{q} = \Delta_{q} - (L_{q} + L_{q}^{\, \prime} )+ D_{q}}
\end{equation}
provided that $\Xi_{q}$ is invertible.

\begin{thm}
\label{ThmCvgR}
Under the causality assumptions on $\cA$ and $\cB$, \textcolor{black}{if we assume that $\Xi_{q}$ is invertible} and that $\dE[V_1^2] = \sigma^2 < \infty$, then we have the almost sure convergence
\begin{equation*}
\limn \whr_{n} = \rho^{*} \cvgas
\end{equation*}
where the limiting value $\rho^{*}$ is given by \eqref{EstRLim}.
\end{thm}
\begin{proof}
See Appendix.
\end{proof}

\noindent Suppose that $q=1$. Then $\Delta_{q} = \ell_0$ and $L_{q} = \theta^{*\, \prime} \Lambda_1^{p} = D_{q}$ (since $\theta^{*} = \Delta_{p}^{-1} \Lambda_1^{p}$ in this particular case) on the one hand, and we have $\Lambda_1^{p} - E_{q} = \ell_1 - \theta^{*\, \prime} ( \Lambda_2^{p} + \Lambda_0^{p} - \Gamma_1 \, \theta^{*})$ on the other hand. That corresponds to Theorem 3.1 of \cite{Proia13} (formulas (C.10) and (C.11) in particular) which states that $\rho^{*} = \theta_{p} \rho\, \theta_{p}^{*}$. Consider now the null hypothesis
\begin{equation*}
\cH_0 : ``\rho_1 = \hdots = \rho_{q} = 0".
\end{equation*}
Then, it is not hard to see that, under $\cH_0$,
\begin{equation}
\label{TLimH0}
\theta^{*} = \theta.
\end{equation}
Accordingly, it follows from the Yule-Walker equations that, for all $h \in \{ 2, \hdots, q+1 \}$,
\begin{equation}
\label{CovH0}
\Lambda_1^{q} = \begin{pmatrix}
\theta^{\, \prime} \Lambda_0^{p} & \theta^{\, \prime} \Lambda_{-1}^{p} & \hdots & \theta^{\, \prime} \Lambda_{-q+1}^{p}
\end{pmatrix}^{\prime} \hsp \text{and} \hsp \Lambda_{h}^{p} = \Gamma_{h-1}\, \theta.
\end{equation}
\begin{cor}
\label{CorCvgRH0}
Suppose that $\cH_0 : ``\rho_1 = \hdots = \rho_{q} = 0"$ is true. Under the causality assumptions on $\cA$, \textcolor{black}{if we assume that $\Xi_{q}$ is invertible} and that $\dE[V_1^2] = \sigma^2 < \infty$, then we have the almost sure convergence
\begin{equation*}
\limn \whr_{n} = 0 \cvgas
\end{equation*}
\end{cor}
\begin{proof}
Under $\cH_0$, the simplifications \eqref{TLimH0} and \eqref{CovH0}, once introduced in \eqref{Eq}, immediately imply that $E_{q}=\Lambda_1^{q}$. Then we conclude using \eqref{EstRLim}.
\end{proof}

\noindent To propose our testing procedure, it only remains to establish the asymptotic normality of $\whr_{n}$ under $\cH_0$. For this purpose, consider the matrix $\Upsilon_{\!pq}$ of order $p \times q$ given, for $p \geq q$, by
\begin{equation}
\label{Ups1}
\Upsilon_{\!pq} = \sigma^2 \begin{pmatrix}
\lambda_0 & 0 & \hdots & \hdots & \hdots & \hdots & 0 \\
\lambda_1 & \ddots & \ddots & & & & \vdots \\
\vdots & \ddots & \ddots & \ddots & & & \vdots \\
\lambda_{q-1} & \hdots & \lambda_1 & \lambda_0 & 0 & \hdots & 0
\end{pmatrix}
\end{equation}
and given, for $p < q$, by
\begin{equation}
\label{Ups2}
\Upsilon_{\!pq} = \sigma^2 \begin{pmatrix}
\lambda_0 & 0 & \hdots & 0 \\
\lambda_1 & \ddots & \ddots & \vdots \\
\vdots & \ddots & \ddots & 0 \\
\lambda_{q-p} & & \ddots & \lambda_0 \\
\vdots & \ddots & & \lambda_1 \\
\vdots & & \ddots & \vdots \\
\lambda_{q-1} & \hdots & \hdots & \lambda_{q-p}
\end{pmatrix}
\end{equation}
where the sequence $\lambda_0, \hdots, \lambda_{q}$ satisfies the recursion
\begin{equation}
\label{SeqLambda}
\left\{
\begin{array}{ccl}
\lambda_0 & = & 1 \\
\lambda_1 & = & \theta_1 \\
\vdots \\
\lambda_{k} & = & \theta_1\, \lambda_{k-1} + \theta_2\, \lambda_{k-2} + \hdots + \theta_{k} \\
\vdots \\
\lambda_{q} & = & \left\{
\begin{array}{ll}
\theta_1\, \lambda_{q-1} + \theta_2\, \lambda_{q-2} + \hdots + \theta_{q} & \hsp \text{for } p \geq q \\
\theta_1\, \lambda_{q-1} + \theta_2\, \lambda_{q-2} + \hdots + \theta_{p}\, \lambda_{q-p} & \hsp \text{for } p < q.
\end{array}
\right.
\end{array}
\right.
\end{equation}

\medskip

\noindent Finally, consider the asymptotic covariance $\Sigma_{\rho}^{\, 0}$ defined as
\begin{equation}
\label{EstRCov}
\Sigma_{\rho}^{\, 0} = I_{q} - \frac{\Upsilon_{\!pq}\, \Delta_{p}^{-1} \Upsilon_{\!pq}^{\, \prime}}{\sigma^2}
\end{equation}
where $\Delta_{p}$ is given in \eqref{Delta} and $\Upsilon_{\!pq}$ in \eqref{Ups1}--\eqref{Ups2}.

\begin{thm}
\label{ThmTlcR}
Suppose that $\cH_0 : ``\rho_1 = \hdots = \rho_{q} = 0"$ is true. Under the causality assumptions on $\cA$, \textcolor{black}{if we assume that $\Xi_{q}$ is invertible} and that $\dE[V_1^4] = \tau^4 < \infty$, then we have the asymptotic normality
\begin{equation*}
\sqrt{n}\, \whr_{n} \liml \cN\big( 0 , \Sigma_{\rho}^{\, 0} \big)
\end{equation*}
where the limiting covariance $\Sigma_{\rho}^{\, 0}$ is given by \eqref{EstRCov}.
\end{thm}
\begin{proof}
See Appendix.
\end{proof}

\noindent Here again, note that $\Sigma_{\rho}^{\, 0}$ does not depend on $\sigma^2$ since a factor $\sigma^{-2}$ is hidden in $\Delta_{p}^{-1}$. For $q=1$, we have
\begin{equation*}
\Upsilon_{\!pq} = \begin{pmatrix}
\sigma^2 & 0 & \hdots & 0
\end{pmatrix}
\end{equation*}
and, after some additional calculations (see \textit{e.g.} the proof of Lemma  D.1 in \cite{Proia13}), we obtain that the first diagonal element $\delta$ of $\Delta_{p}^{-1}$ is given by
\begin{equation*}
\delta = \frac{1 - \theta_{\!p}^{\, 2}}{\sigma^2}.
\end{equation*}
Hence, $\Sigma_{\rho}^{\, 0} = \theta_{\!p}^{\, 2}$ (a result that may be found in \cite{BercuProia13}--\cite{Proia13}). In the last section, we present our statistical testing procedure.


\section{A statistical testing procedure}
\label{SecStatProc}


First, there is a set $\Omega^{*}$ of pathological cases that we will not study here. They correspond to the situations where $\rho^{*} = 0$ under $\cH_1 : ``\exists\, h \in \{ 1, \hdots, q \},\, \rho_{h} \neq 0"$ or where $\Sigma_{\rho}^{\, 0}$ or $\Xi_{q}$ is not invertible. When $q=1$, we have $\Omega^{*} = \{ \theta_{p} - \theta_{p-1} \rho = \theta_{p}\, \rho\, (\theta_1 + \rho) \} \cap \{\theta_{p} = 0 \}$ for $p \geq 1$, and simply $\Omega^{*} = \{ \theta = -\rho \} \cap \{\theta = 0 \}$ for $p=1$. Under our assumption that $\theta_{p} \neq 0$, $\Omega^{*}$ is a set of particular cases that do not restrain at all the whole study. Assume that we want to test
\begin{equation*}
\cH_0 : ``\rho_1 = \hdots = \rho_{q} = 0" \hsp \text{against} \hsp \cH_1 : ``\exists\, h \in \{ 1, \hdots, q \},\, \rho_{h} \neq 0"
\end{equation*}
for a given $q \geq 1$ such that $\{ \theta, \rho \} \notin \Omega^{*}$. Consider the test statistic given by
\begin{equation}
\label{StatTest}
\wh{T}_{n} ~ = ~ n\, \whr_{n}^{~\, \prime} \left( I_{q} - \frac{\wh{P}_{n}\, S_{n}^{-1}\, \wh{P}_{n}^{\, \prime}}{n\, \wh{\sigma}_{n}^{\, 2}} \right)^{\!-1}\! \whr_{n}
\end{equation}
where, for all $n \geq 1$,
\begin{equation*}
\wh{\sigma}_{n}^{\, 2} = \frac{1}{n} \sum_{t=1}^{n} \wh{Z}_{t}^{\: 2} \hsp \text{and} \hsp \wh{P}_{n} = \sum_{t=1}^{n} \wh{\Psi}_{t}^{q}\, \Phi_{t}^{p\: \prime}
\end{equation*}
with the whole notations above.

\begin{thm}
\label{ThmStatTest}
Assume that $\{ \theta, \rho \} \notin \Omega^{*}$. If $\cH_0 : ``\rho_1 = \hdots = \rho_{q} = 0"$ is true, then under the causality assumptions on $\cA$ and as soon as $\dE[V_1^4] = \tau^4 < \infty$, we have
\begin{equation*}
\wh{T}_{n} \liml \chi^2_{q}
\end{equation*}
where $\wh{T}_{n}$ is the test statistic given in \eqref{StatTest} and $\chi^2_{q}$ has a chi-square distribution with $q$ degrees of freedom. In addition, if $\cH_1 : ``\exists\, h \in \{ 1, \hdots, q \},\, \rho_{h} \neq 0"$ is true such that $\cB$ is causal, then
\begin{equation*}
\limn \vert\, \wh{T}_{n} \vert = +\infty \cvgas
\end{equation*}
\end{thm}
\begin{proof}
Under $\cH_0$, $\wh{\sigma}_{n}^{\, 2}$ is a consistent estimator of $\sigma^2$ and it is not hard to see, as it is done in the proof of Theorem \ref{ThmCvgR}, that $\wh{P}_{n}/n$ is a consistent estimator of $\Upsilon_{\!pq}$. The end of the proof follows from Theorem \ref{ThmTlcR}, so we leave it to the reader.
\end{proof}

\noindent From a practical point of view, for a significance level $0 < \alpha < 1$, the rejection region is given by $\cR = ~ ]z_{q,\, 1-\alpha}, +\infty[$ where $z_{q,\, 1-\alpha}$ stands for the $(1-\alpha)-$quantile of the chi-square distribution with $q$ degrees of freedom. The null hypothesis $\cH_0$ will be rejected if the empirical value satisfies
\begin{equation}
\label{RejectRule}
\vert\, \wh{T}_{n} \vert > z_{q,\, 1-\alpha}.
\end{equation}

\noindent Let us now compare the empirical power of this procedure with the commonly used portmanteau tests of Ljung-Box \cite{LjungBox78} and Box-Pierce \cite{BoxPierce70}, and with the Breusch-Godfrey \cite{Breusch78}--\cite{Godfrey78} procedure. \textcolor{black}{Of course it would deserve an in-depth simulation study, we only give here an overview of the results obtained on some examples.} Two arbitrarily chosen causal autoregressive processes $(Y_{1,\, n})$ and $(Y_{2,\, n})$ are generated (with $p=3$, $\theta_1 = \frac{3}{10}$, $\theta_2 = -\frac{1}{5}$ and $\theta_3 = \frac{2}{5}$ in the first case, $p=2$, $\theta_1 = \frac{17}{10}$ and $\theta_2 = -\frac{18}{25}$ in the second case), with initial values set to zero, using different driven noises $(Z_{1,\, n})$ and $(Z_{2,\, n})$. Each of them is itself generated following \eqref{ModARGen}, for different values of $q \in \{1, \hdots, 4\}$ \textcolor{black}{and $\rho$}, such that
\begin{equation*}
(V_{1,\, n}) \iid \cU([-2, 2]) \hsp \text{and} \hsp (V_{2,\, n}) \iid \cN(0, 2).
\end{equation*}
For a simulated path, we only test for serial correlation using our procedure (Thm \ref{ThmStatTest}), the Breusch-Godfrey procedure (BG) and the Ljung-Box procedure (LB), since we know that the Box-Pierce one is an equivalent of LB, for $q^{\, \prime} = \{1, \hdots, 4\}$ \textcolor{black}{and $\alpha = 0.05$}. For $N=10^4$ simulations of each configuration, the frequency with which $\cH_0$ is rejected yields an estimator of the power of the procedures, that is
\begin{equation*}
\wh{\dP}\big( \text{reject } \cH_0\, \vert\, \cH_1 \text{ is true}\big).
\end{equation*}
\noindent We repeat the simulations for $n=3000$, $n=300$ and $n=30$, to have an overview of the results in an asymptotic context as well as on reasonable and small-sized samples. Finally, we also test the particular cases where $q=0$ (with $q^{\, \prime} = \{1, \hdots, 4\}$) to evaluate the behavior of the procedures under the null of serial independance. All our results are summarized on Figures \ref{Fig_Hist1} and \ref{Fig_Pow1} for the first model, on Figures \ref{Fig_Hist2} and \ref{Fig_Pow2} for the second model. (Thm \ref{ThmStatTest}) is represented in blue, (BG) in orange and (LB) in violet. For each of them, the solid line is associated with $q^{\, \prime} = 1$, the dotted line with $q^{\, \prime} = 2$, the dashed line with $q^{\, \prime} = 3$ and the two-dashed line with $q^{\, \prime} = 4$.

\begin{figure}[h!]
\includegraphics[width=15.5cm]{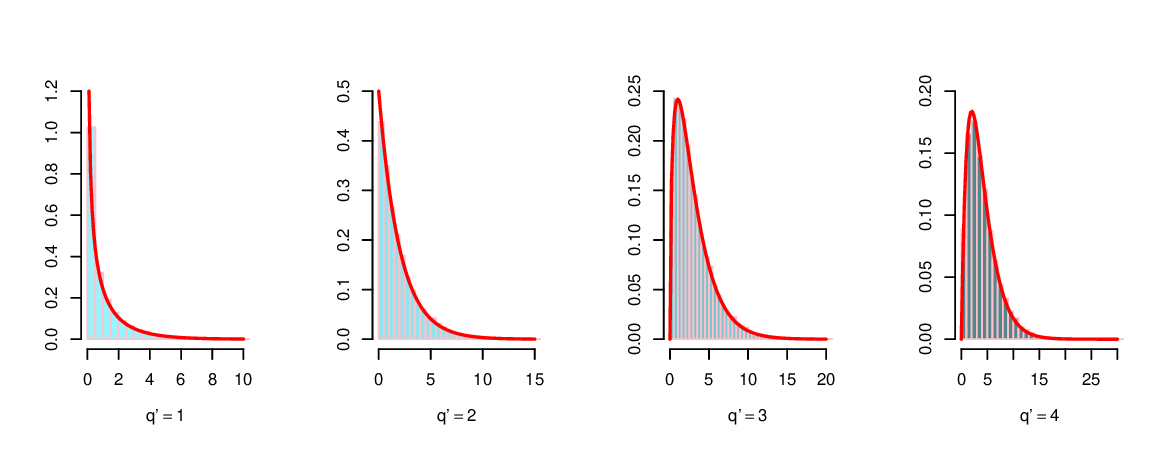}
\vspace{-0.5cm} \\
\caption{Empirical distribution of the test statistic under the null in the first model for $q^{\, \prime} = \{1, \hdots, 4\}$ and $n=3000$, superimposition of the theoretical limiting distributions.}
\label{Fig_Hist1}
\end{figure}

\begin{figure}[h!]
\includegraphics[width=15.5cm]{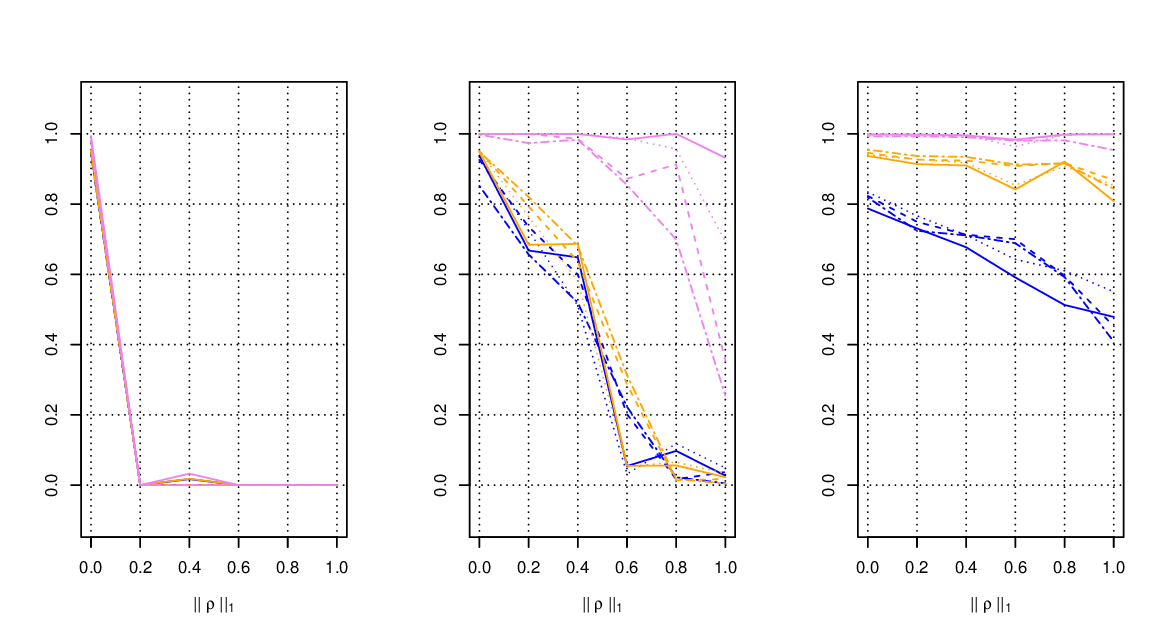}
\vspace{-0.5cm} \\
\caption{Frequency of non-rejection of the null for $n=3000$ (left), $n=300$ (center) and $n=30$ (right), in the first model.}
\label{Fig_Pow1}
\end{figure}

\begin{figure}[h!]
\includegraphics[width=15.5cm]{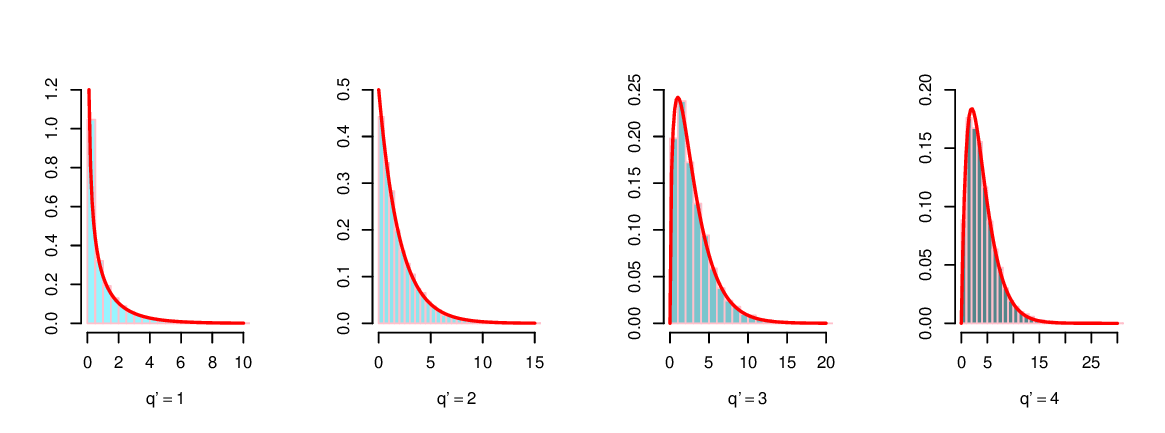}
\vspace{-0.5cm} \\
\caption{Empirical distribution of the test statistic under the null in the second model for $q^{\, \prime} = \{1, \hdots, 4\}$ and $n=3000$, superimposition of the theoretical limiting distributions.}
\label{Fig_Hist2}
\end{figure}

\begin{figure}[h!]
\includegraphics[width=15.5cm]{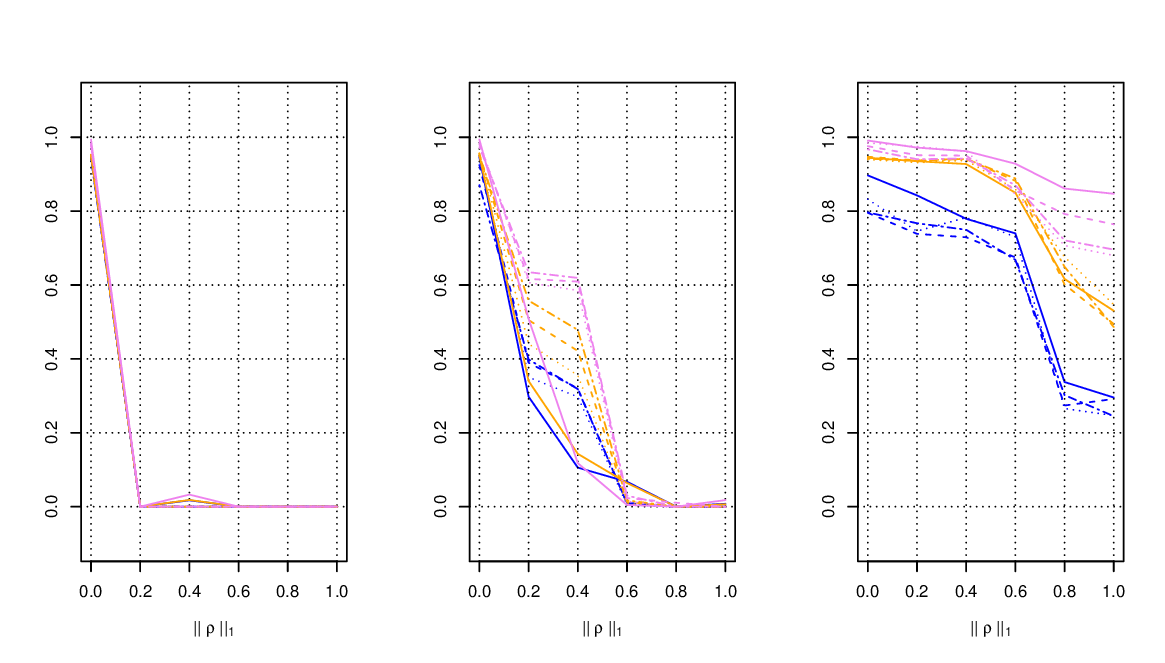}
\vspace{-0.5cm} \\
\caption{Frequency of non-rejection of the null for $n=3000$ (left), $n=300$ (center) and $n=30$ (right), in the second model.}
\label{Fig_Pow2}
\end{figure}

\noindent The histograms built under $\cH_0$ justifies the $1-\alpha$ level of the test, since one can observe that there is a strong adequation between the empirical test statistics and their theoretical distributions. Conversely, under $\cH_1$, our procedure appears to be more powerful on samples of reasonable size, and especially on small-sized samples. Giving conclusions on the basis of two examples seems particularly precarious, but we have of course considered lots of other examples that we cannot display in this paper, for larger $p$ and $q$, \textcolor{black}{for various generating processes and different values of $\theta$ and $\rho$ (satisfying the causality assumptions)}.


\section{Concluding remarks}


We have provided a sharp analysis on the asymptotic behavior of the least squares estimator of the parameter of an autoregressive process when the noise is also driven by an autoregressive process, establishing results that generalize \cite{BercuProia13} and \cite{Proia13}, also weakening assumptions. In addition, the investigation of the residual set yielded an estimator of the serial correlation parameter together with its limiting value and its asymptotic normality, under the null. We have observed that our procedure gave either better or equivalent results than the usual tests for serial correlation, depending on the sample size. Due to calculation complexity, we have only stipulated the asymptotic normality of the serial correlation estimator under the null. Even if it is of lesser usefulness on real data, it could be challenging to generalize Theorem \ref{ThmTlcR} to the alternative. Nevertheless, this study leads to a particular case of great practical interest. Suppose that the estimation of the autoregressive part of the process is misjudged and done for $p^{\, *} < p$. Then, for all $z \in \dC$, we have the polynomial expression
\begin{equation*}
\cA(z)\, \cB(z) ~ = ~ \prod_{k=1}^{p} \Big(1 - \frac{z}{\lambda_{k}} \Big)\, \cB(z) ~ = ~ \prod_{k=1}^{p^{\, *}} \Big(1 - \frac{z}{\lambda_{k}} \Big)\, \cB^{\, *}(z) ~ = ~ \cA^{*}(z)\, \cB^{\, *}(z)
\end{equation*}
where $(\lambda_{k})$ are the roots of $\cA$, $\cA^{*}$ is a causal polynomial of order $p^{\, *}$ and $\cB^{\, *}$ is a causal polynomial of order $q+p-p^{\, *}$. Similarly, if the estimation of the autoregressive part of the process is misjudged with $p^{\, *} > p$, there exists causal polynomials $\cA^{*}$ and $\cB^{\, *}$ of order $p^{\, *}$ and $q+p-p^{\, *}$, respectively, such that $\cA(z)\, \cB(z) = \cA^{*}(z)\, \cB^{\, *}(z)$. If $p^{\, *} \geq p+q$, then there is no serial correlation in $\cB$ anymore and we are placed under the null. This little reasoning shows that $p$ does not play a crucial role in this procedure, which is a corollary of great practical interest considering the issue of selecting the minimal value of $p$ in an autoregressive modelling. \textcolor{black}{The procedure can accordingly be used to test for the true value of $p$ in an AR$(p)$ model, the smallest one not leading to $\cH_1$.} Of course, a substantial contribution would be to drive all these results to general ARMA processes and a lot of pathological cases also remain to study (namely, the full description of $\Omega^{*}$).

\smallskip

To conclude, let us mention that once detected, several algorithms exist to produce unbiased estimators despite the residual correlation (see \textit{e.g.} Pierce \cite{Pierce70} and Hatanaka \cite{Hatanaka74}). Theses estimators have usefully to be compared to their biased counterparts (see \textit{e.g.} Sargent \cite{Sargent68}, Hong and L'Esperance \cite{HongLEsperance73},
Maeshiro \cite{Maeshiro80}--\cite{Maeshiro87}--\cite{Maeshiro90}--\cite{Maeshiro96}--\cite{Maeshiro99}, Flynn and Westbrook \cite{FlynnWestbrook84}, or Jekanowski and Binkley \cite{JekanowskiBinkley96}). We can also cite the recent approaches of Francq, Roy and Zako\"ian \cite{FrancqRoyZakoian05} in 2005, and Duchesne and Francq \cite{DuchesneFrancq08} in 2008, where in particular the aysmptotic covariance matrix of a vector of autocorrelations for residuals is estimated in an ARMA model to get the asymptotic distribution of the portmanteau statistics. The objective is also to correct the bias generated by the presence of residual correlation.

\smallskip

\noindent \textcolor{black}{\textbf{Acknowledgments.} We thank the anonymous Reviewer for the suggestions and comments.}


\section*{Technical appendix}


\noindent \textbf{Proof of Lemma \ref{LemCausal}.} \textcolor{black}{From \eqref{ModARGen}, we directly get}
\begin{equation}
\label{ModARPol}
\cB(L)\cA(L) Y_{t} = V_{t}
\end{equation}
which shows that $(Y_{t})$ is an AR($p+q$) process. Moreover, the relation
\begin{equation*}
(1 - \rho_1\, z - \hdots - \rho_{q}\, z^{q})(1 - \theta_1\, z - \hdots - \theta_{p}\, z^{p}) = 1 - \beta_1\, z - \hdots - \beta_{p+q}\, z^{p+q}
\end{equation*}
easily enables to identify $\beta$ in \eqref{Beta}. The causality assumptions on $\cA$ and $\cB$ directy implies the causality of the product polynomial $\cB \cA$. From Theorem 3.1.1 of \cite{BrockwellDavis91} associated with Remark 2 that follows, we know that \eqref{ModARPol} has the MA($\infty$) expression given, for all $t \in \dZ$, by
\begin{equation}
\label{ModMAInf}
Y_{t} = \sum_{k=0}^{\infty} \psi_{k} V_{t-k} \hsp \text{with} \hsp \sum_{k=0}^{\infty} \vert\, \psi_{k} \vert < \infty.
\end{equation}
Hence, we define the function from $\dR^{\infty}$ into $\dR$ as
\begin{equation*}
g(x_0,\, x_1,\, \hdots) = \sum_{k=0}^{\infty} \psi_{k}\, x_{k}
\end{equation*}
and we note that, for all $x, y \in \dR^{\infty}$,
\begin{equation*}
\vert\, g(x) - g(y) \vert ~ \leq ~ \sum_{k=0}^{\infty} \vert\, \psi_{k} \vert \vert\, x_{k} - y_{k} \vert ~ \leq ~ \Vert x - y \Vert_{\infty} \sum_{k=0}^{\infty} \vert\, \psi_{k} \vert.
\end{equation*}
We deduce that $g$ is Lipschitz, and thus continuous. By virtue of Theorem 3.5.8 of \cite{Stout74} and since $(V_{t})$ is obviously an ergodic process, we conclude from \eqref{ModMAInf} that $(Y_{t})$ is also an ergodic process.
\hfill \qedsymbol

\bigskip

\noindent \textbf{Proof of Lemma \ref{LemInvD}.} To prove that $\Delta_{h}$ given in \eqref{Delta} is positive definite for all $h \geq 1$, we will use a methodology very similar to the one establishing Lemma 2.2 in \cite{Proia13}. As a matter of fact, we know from Lemma \ref{LemCausal} that the polynomial $\cB \cA$ in \eqref{ModARPol} is causal. As a consequence, from Theorem 4.4.2 of \cite{BrockwellDavis91}, the spectral density $f_{Y}$ associated with $(Y_{t})$ is given, for all $x$ in the torus $\dT = [-\pi,\pi]$, by
\begin{equation*}
f_{Y}(x) = \frac{\sigma^2}{2 \pi \vert \cB(\de^{-\di x}) \cA(\de^{-\di x}) \vert}.
\end{equation*}
In addition, it is well-known that the covariance matrix of order $h$ of the stationary process $(Y_{t})$ coincides with the Toeplitz matrix of order $h$ of the Fourier coefficients associated with $f_{Y}$. Namely, for all $h \geq 1$,
\begin{equation*}
\cT_{h}(f_{Y}) = \big( \wh{f}_{i-j} \big)_{\!1\, \leq\: i,\,j\, \leq\: h} = \Delta_{h}
\end{equation*}
where $\cT_{h}$ is the Toeplitz operator of order $h$ and, for all $k \in \dZ$,
\begin{equation*}
\wh{f}_{k} = \int_{\dT} f_{Y}(x)\, \de^{-\di k x}\, \dd x.
\end{equation*}
Finally, we apply Proposition 4.5.3 of \cite{BrockwellDavis91} (see also \cite{GrenanderSzego58}) to conclude that, since we obviously have
\begin{equation*}
0 < \inf_{x\, \in\: \dT}\, f_{Y}(x),
\end{equation*}
then
\begin{equation*}
0 ~ < ~ 2 \pi \inf_{x\, \in\: \dT}\, f_{Y}(x) ~ \leq ~ \lambda_{\min}(\Delta_{h}) ~ \leq ~ \lambda_{\max}(\Delta_{h}).
\end{equation*}
This clearly shows that for all $h \geq 1$, $\Delta_{h}$ is positive definite and thus invertible.
\hfill \qedsymbol

\bigskip

\noindent \textbf{Proof of Proposition \ref{PropCvgTn}.} At this stage, one needs an additional technical lemma, directly exploiting the ergodicity of the process.

\begin{lem}
\label{LemCvgSumY2}
Under the causality assumptions on $\cA$ and $\cB$ and as soon as $\dE[V_1^2] = \sigma^2 < \infty$, we have the almost sure convergence, for all $h \in \{ 0, \hdots, p+q \}$,
\begin{equation*}
\limn \frac{1}{n} \sum_{t=1}^{n} Y_{t-h} Y_{t} = \ell_{h} \cvgas
\end{equation*}
where $\ell_0, \hdots, \ell_{p+q}$ are given in \eqref{VecCov}. In particular, we also have the almost sure convergence
\begin{equation*}
\limn \frac{Y_{n-h} Y_{n}}{n} = 0 \cvgas
\end{equation*}
\end{lem}
\begin{proof}
As soon as $\dE[V_1^2] = \sigma^2 < \infty$, this is an immediate consequence of the ergodicity of $(Y_{t})$, stipulated in Lemma \ref{LemCausal}.
\end{proof}

\noindent From \eqref{ModAR} and the associated notations, we deduce that, for all $1 \leq t \leq n$,
\begin{equation*}
Y_{t} = \alpha^{\, \prime} \Phi_{t-1}^{p} + \gamma^{\, \prime} \Phi_{t-p-1}^{q} + V_{t}
\end{equation*}
where $\alpha$ of order $p$ and $\gamma$ of order $q$ are defined in \eqref{AlphaGamma}. It follows that, for all $h \in \{ 1, \hdots, q \}$,
\begin{equation}
\label{SumAR}
\sum_{t=1}^{n} \Phi_{t-h}^{p} Y_{t} = \sum_{t=1}^{n} \Phi_{t-h}^{p} \Phi_{t-1}^{p\,\, \prime} \alpha + \sum_{t=1}^{n} \Phi_{t-h}^{p} \Phi_{t-p-1}^{q\,\, \prime} \gamma + \sum_{t=1}^{n} \Phi_{t-h}^{p} V_{t}.
\end{equation}
Note also, for all $k \in \{ 0, \hdots, p \}$ and $h \in \{ 1, \hdots, q \}$,
\begin{equation*}
\Sigma_{n\!,\, k} = \sum_{t=1}^{n} \Phi_{t}^{p}\, Y_{t-k} \hsp \text{and} \hsp \Pi_{n\!,\, h} = \sum_{t=1}^{n} \Phi_{t-h}^{p} Y_{t}.
\end{equation*}
To lighten the calculations, we will now omit the subscripts (from $t=1$ to $n$) on the summation operators. We will also suppose that $p \geq q$. Then, for $h=1$ and $h=2$, the first term in the right-hand side of \eqref{SumAR} is built using
\begin{eqnarray}
\sum \Phi_{t-1}^{p} \Phi_{t-1}^{p\,\, \prime} \alpha & = & \sum \Phi_{t-1}^{p} (\alpha_1 Y_{t-1} + \alpha_2 Y_{t-2} + \hdots + \alpha_{p} Y_{t-p} )\nonumber \\
 & = & \alpha_1 \sum \Phi_{t-1}^{p} Y_{t-1} + \alpha_2 \sum \Phi_{t-1}^{p} Y_{t-2} + \hdots + \alpha_{p} \sum \Phi_{t-1}^{p} Y_{t-p}\nonumber \\
 & = & \alpha_1\, \Sigma_{n\!,\, 0} + \alpha_2\, \Sigma_{n\!,\, 1} + \hdots + \alpha_{p}\, \Sigma_{n\!,\, p-1} + r_{1,\, n},
\label{EstTDev1}
\end{eqnarray}
\begin{eqnarray}
\sum \Phi_{t-2}^{p} \Phi_{t-1}^{p\,\, \prime} \alpha & = & \sum \Phi_{t-2}^{p} (\alpha_1 Y_{t-1} + \alpha_2 Y_{t-2} + \hdots + \alpha_{p} Y_{t-p} )\nonumber \\
 & = & \alpha_1 \sum \Phi_{t-2}^{p} Y_{t-1} + \alpha_2 \sum \Phi_{t-2}^{p} Y_{t-2} + \hdots + \alpha_{p} \sum \Phi_{t-2}^{p} Y_{t-p}\nonumber \\
 & = & \alpha_1\, \Pi_{n\!,\, 1} + \alpha_2\, \Sigma_{n\!,\, 0} + \hdots + \alpha_{p}\, \Sigma_{n\!,\, p-2} + r_{2,\, n}.
\label{EstTDev2}
\end{eqnarray}
\medskip
\noindent Following the same reasoning for each $h$, we finally obtain, for $h = q$,
\begin{eqnarray}
\sum \Phi_{t-q}^{p} \Phi_{t-1}^{p\,\, \prime} \alpha & = & \sum \Phi_{t-q}^{p} (\alpha_1 Y_{t-1} + \alpha_2 Y_{t-2} + \hdots + \alpha_{p} Y_{t-p} )\nonumber \\
 & = & \alpha_1 \sum \Phi_{t-q}^{p} Y_{t-1} + \alpha_2 \sum \Phi_{t-q}^{p} Y_{t-2} + \hdots + \alpha_{p} \sum \Phi_{t-q}^{p} Y_{t-p}\nonumber \\
 & = & \alpha_1\, \Pi_{n\!,\, q-1} + \alpha_2\, \Pi_{n\!,\, q-2} + \hdots + \alpha_{p}\, \Sigma_{n\!,\, p-q} + r_{q,\, n}.
\label{EstTDev4}
\end{eqnarray}
\medskip
\noindent The remainder terms satisfy, \text{via} Lemma \ref{LemCvgSumY2} and for all $h \in \{1, \hdots, q \}$, 
\begin{equation}
\label{CvgResS1}
\Vert r_{h,\, n} \Vert = o(n) \cvgas
\end{equation}
We will now study in the same manner the more intricate second term in \eqref{SumAR}. For the first values of $h$, we have the following equalities.
\begin{eqnarray}
\sum \Phi_{t-1}^{p} \Phi_{t-p-1}^{q\,\, \prime} \gamma & = & \sum \Phi_{t-1}^{p} (\gamma_1 Y_{t-p-1} + \gamma_2 Y_{t-p-2} + \hdots + \gamma_{q} Y_{t-p-q} )\nonumber \\
 & = & \!\gamma_1 \sum \Phi_{t-1}^{p} Y_{t-p-1} + \gamma_2 \sum \Phi_{t-1}^{p} Y_{t-p-2} + \hdots + \gamma_{q} \sum \Phi_{t-1}^{p} Y_{t-p-q}\nonumber \\
 & = & \gamma_1\, J_{p}\, \Pi_{n\!,\, 1} + \gamma_2\, J_{p}\, \Pi_{n\!,\, 2} + \hdots + \gamma_{q}\, J_{p}\, \Pi_{n\!,\, q} + \tau_{1,\, n}.
\label{EstTDev5}
\end{eqnarray}
\begin{eqnarray}
\sum \Phi_{t-2}^{p} \Phi_{t-p-1}^{q\,\, \prime} \gamma & = & \sum \Phi_{t-2}^{p} (\gamma_1 Y_{t-p-1} + \gamma_2 Y_{t-p-2} + \hdots + \gamma_{q} Y_{t-p-q} )\nonumber \\
 & = & \!\gamma_1 \sum \Phi_{t-2}^{p} Y_{t-p-1} + \gamma_2 \sum \Phi_{t-2}^{p} Y_{t-p-2} + \hdots + \gamma_{q} \sum \Phi_{t-2}^{p} Y_{t-p-q}\nonumber \\
 & = & \gamma_1\, \Sigma_{n\!,\, p-1} + \gamma_2\, J_{p}\, \Pi_{n\!,\, 1} + \hdots + \gamma_{q}\, J_{p}\, \Pi_{n\!,\, q-1} + \tau_{2,\, n}.
\label{EstTDev6}
\end{eqnarray}
\medskip
\noindent Finally, for $h = q$,
\begin{eqnarray}
\sum \Phi_{t-q}^{p} \Phi_{t-p-1}^{q\,\, \prime} \gamma & = & \sum \Phi_{t-q}^{p} (\gamma_1 Y_{t-p-1} + \gamma_2 Y_{t-p-2} + \hdots + \gamma_{q} Y_{t-p-q} )\nonumber \\
 & = & \!\gamma_1 \sum \Phi_{t-q}^{p} Y_{t-p-1} + \gamma_2 \sum \Phi_{t-q}^{p} Y_{t-p-2} + \hdots + \gamma_{q} \sum \Phi_{t-q}^{p} Y_{t-p-q}\nonumber \\
 & = & \gamma_1\, \Sigma_{n\!,\, p-q+1} + \gamma_2\, \Sigma_{n\!,\, p-q+2} + \hdots + \gamma_{q}\, J_{p}\, \Pi_{n\!,\, 1} + \tau_{q,\, n}.
\label{EstTDev8}
\end{eqnarray}
\medskip
\noindent Again, the remainder terms satisfy, \text{via} Lemma \ref{LemCvgSumY2} and for all $h \in \{1, \hdots, q \}$, 
\begin{equation}
\label{CvgResS2}
\Vert \tau_{h,\, n} \Vert = o(n) \cvgas
\end{equation}
The following step is to note that
\begin{equation*}
S_{n-1} = \begin{pmatrix}
\Sigma_{n\!,\, 0} & \Sigma_{n\!,\, 1} & \hdots & \Sigma_{n\!,\, p-1}
\end{pmatrix} + o(n) \cvgas
\end{equation*}
and that
\begin{equation*}
P_{n} = \begin{pmatrix}
\Pi_{n\!,\, 1} & \Pi_{n\!,\, 2} & \hdots & \Pi_{n\!,\, q}
\end{pmatrix} + o(n) \cvgas
\end{equation*}
where $S_{n}$ and $P_{n}$ are given in \eqref{Sn} and \eqref{Pn}, respectively. Combining all these equations, it is now easy to establish that, for all $n \geq 1$,
\begin{equation}
\label{EstTVecEq}
P_{n} = S_{n-1}\, D + P_{n}\, C_{\alpha}^{\, \prime} + J_{p}\, P_{n}\, C_{\gamma} + M_{n} + \xi_{n}
\end{equation}
where $C_{\alpha}$, $C_{\gamma}$ and $D$ are given in \eqref{CaCg} and \eqref{D}, where $\xi_{n}$ is a residual such that $\VVert \xi_{n} \VVert = o(n)$ a.s. and where $M_{n}$ is the following matrix martingale of order $p \times q$,
\begin{equation}
\label{MartMat}
M_{n} = \begin{pmatrix}
M_{1,\, n} & M_{2,\, n} & \hdots & M_{q,\, n}
\end{pmatrix}
\end{equation}
in which each vector is itself a vector martingale of order $p$ given, for all $h \in \{1, \hdots, q \}$, by
\begin{equation*}
M_{h,\, n} = \sum_{t=1}^{n} \Phi_{t-h}^{p} V_{t}.
\end{equation*}
Under our assumptions, the vector martingale $M_{h,\, n}$ is locally square-integrable and adapted to the filtration $(\cF_{n})$ defined in the beginning of Section \ref{SecEstT}. Its predictable quadratic variation is given, for all $n \geq 1$, by
\begin{equation}
\label{MartCroch}
\langle M_{h} \rangle_{n} = \sum_{t=1}^{n} \dE[ (\Delta M_{h,\, t})(\Delta M_{h,\, t})^{\prime}\, \vert\, \cF_{t-1} ] = \sigma^2 (S_{n-h} - S)
\end{equation}
and obviously satisfies, by virtue of Lemma \ref{LemCvgSumY2},
\begin{equation*}
\limn \frac{\langle M_{h} \rangle_{n}}{n} = \sigma^2 \Delta_{p} \cvgas
\end{equation*}
where $\Delta_{p}$ is the covariance matrix of order $p$ given in \eqref{Delta}. Whence we obtain that
\begin{equation*}
\limn \frac{\text{tr}(\langle M_{h} \rangle_{n})}{n} = \sigma^2\, p\, \ell_0 > 0 \cvgas
\end{equation*}
Since $\Delta_{p}$ is positive definite (Lemma \ref{LemInvD}), $\lambda_{\text{min}}(\langle M_{h} \rangle_{n})$ diverges and, for any $\delta > 0$,
\begin{equation*}
\log(\lambda_{\text{max}}(\langle M_{h} \rangle_{n}))^{1+\delta} = o(\lambda_{\text{min}}(\langle M_{h} \rangle_{n}) \cvgas
\end{equation*}
This is sufficient to apply the strong law of large numbers for vector martingales (Theorem 4.3.15 of \cite{Duflo97}, or \cite{DufloSenoussiTouati90}). Thus,
\begin{equation*}
\limn \langle M_{h} \rangle_{n}^{-1} M_{h,\, n} = 0 \cvgas
\end{equation*}
since $S_{n-h} = S_{n-1} + o(n)$ a.s. Reasoning column by column, it follows from all our previous calculations that
\begin{equation}
\label{EstTCvgMartRes}
\limn \langle M_{h} \rangle_{n}^{-1} M_{n} = 0 \cvgas \hsp \text{and} \hsp \limn \langle M_{h} \rangle_{n}^{-1} \xi_{n} = 0 \cvgas
\end{equation}
where $M_{n}$ is the matrix martingale given by \eqref{MartMat}. Let us get back to \eqref{EstTVecEq}. For all $n \geq 1$,
\begin{equation}
\label{EstTVecEq2}
T_{n}\,(I_{q} - C_{\alpha}^{\, \prime}) - J_{p}\, T_{n}\, C_{\gamma} = D + S_{n-1}^{-1}\, M_{n} + S_{n-1}^{-1}\, R_{n}
\end{equation}
if we remember that $T_{n} = S_{n-1}^{-1} P_{n}$, and if we note that $J_{p}\, S_{n-1}^{-1} = S_{n-1}^{-1}\, J_{p} + O(n^{-1})$ a.s., which implies to slightly modify $\xi_{n}$ into $R_{n}$. Now, we use the combination of Lemmas \ref{LemInvD} and \ref{LemCvgSumY2} to establish that
\begin{equation*}
\limn T_{n} = \Delta_{p}^{-1}\, \Pi_{pq} = T^{*} \cvgas
\end{equation*}
using the notations of Proposition \ref{PropCvgTn}. \textit{Via} \eqref{MartCroch}, \eqref{EstTCvgMartRes} and \eqref{EstTVecEq2},
\begin{equation}
\label{EstTVecEq3}
T^{*}\,(I_{q} - C_{\alpha}^{\, \prime}) - J_{p}\, T^{*}\, C_{\gamma} = D.
\end{equation}
The latter relation is a generalized Sylvester matrix equation, deeply studied for example in \cite{Chu87}. From Theorem 1 of the same reference, we know that \eqref{EstTVecEq3} has a unique solution for $T^{*}$ if and only if $I_{p} - \lambda J_{p}$ and $C_{\gamma} - \lambda (I_{q} - C_{\alpha}^{\, \prime})$ are regular matrix pencils (which is obvious here) and if $\det(C_{\gamma} - \lambda (I_{q} - C_{\alpha}^{\, \prime})) \neq 0$ for $\vert \lambda \vert = 1$. As a matter of fact, it is not hard to see that $\det(I_{p} - \lambda J_{p}) = 0$ is only satisfied for $\vert \lambda \vert = 1$. Assuming that a solution exists, it is given by
\begin{equation}
\label{EstTSol}
\vec(T^{*}) = \big( (I_{q} - C_{\alpha}) \otimes I_{p}  -  C_{\gamma} \otimes J_{p} \big)^{-1}\, \vec(D) = K^{-1}\, \vec(D).
\end{equation}
Since the study of $\det(C_{\gamma} - \lambda (I_{q} - C_{\alpha}^{\, \prime}))$ seems far too complicated for any dimensions $p$ and $q$, we only \textcolor{black}{assume the invertibility of $K$.} We will not detail the case where $p < q$ since the reasoning is very similar. Indeed, it is enough to replace $C_{\alpha}$ by $G_{\alpha}$ and $D$ by $E$ to get the same results \textit{via} the same lines.
\hfill \qedsymbol

\bigskip

\noindent \textbf{Proof of Proposition \ref{PropTlcTn}.} For this proof, one also needs some ergodic properties on the fourth-order moments of the process.

\begin{lem}
\label{LemCvgSumY4}
Under the causality assumptions on $\cA$ and $\cB$ and as soon as $\dE[V_1^4] = \tau^4 < \infty$, we have the almost sure convergence
\begin{equation*}
\limn \frac{1}{n} \sum_{t=1}^{n} Y_{t}^4 = \kappa^4 < \infty \cvgas
\end{equation*}
In particular, we also have the almost sure convergence
\begin{equation*}
\limn \frac{Y_{n}^4}{n} = 0 \cvgas
\end{equation*}
\end{lem}
\begin{proof}
As soon as $\dE[V_1^4] = \tau^4 < \infty$, this is an immediate consequence of the ergodicity of $(Y_{t})$, stipulated in Lemma \ref{LemCausal}. As a matter of fact, for the stationary process defined on $\dZ$ and using the notations of \eqref{ModMAInf}, there exists $\delta < \infty$ such that 
\begin{eqnarray*}
\dE[Y_{t}^4] & = & \sum_{i=0}^{\infty} \sum_{j=0}^{\infty} \sum_{k=0}^{\infty} \sum_{\ell=0}^{\infty} \psi_{i}\, \psi_{j}\, \psi_{k}\, \psi_{\ell}\, \dE[V_{t-i} V_{t-j} V_{t-k} V_{t-\ell}] \\
 & = & \tau^4 \sum_{k=0}^{\infty} \psi_{k}^4 + \delta \sigma^4 \sum_{k=0}^{\infty} \left[ \sum_{\ell = 0}^{k-1} + \sum_{\ell = k+1}^{\infty} \right] \psi_{k}^2\, \psi_{\ell}^2  = \kappa^4 < \infty
\end{eqnarray*}
since we have already seen that $(\psi_{k})$ is absolutely summable.
\end{proof}
\noindent From \eqref{EstTVecEq2}, we know that, for all $n \geq 1$ and $p \geq q$,
\begin{equation*}
T_{n}\,(I_{q} - C_{\alpha}^{\, \prime}) - J_{p}\, T_{n}\, C_{\gamma} = D + S_{n-1}^{-1}\, M_{n} + S_{n-1}^{-1}\, R_{n}
\end{equation*}
with the notations above. Hence, \textcolor{black}{if $K$ is invertible,} the resolution of this generalized Sylvester matrix equation \cite{Chu87} leads to
\begin{equation}
\label{EstTDecomp}
\vec(T_{n}) = K^{-1}\, \vec(D  + S_{n-1}^{-1}\, M_{n} + S_{n-1}^{-1}\, R_{n}).
\end{equation}
It follows that
\begin{equation}
\label{EstTVecEqTlc}
\sqrt{n}\, \big( \vec(T_{n}) - \vec(T^{*}) \big) = \sqrt{n}\, K^{-1}\, \vec(S_{n-1}^{-1}\, M_{n}) + \sqrt{n}\, K^{-1}\, \vec(S_{n-1}^{-1}\, R_{n})
\end{equation}
where we recall that $M_{n}$ is the matrix $(\cF_{n})$--martingale given by \eqref{MartMat}. First, from the combination of Lemmas \ref{LemCvgSumY2} and \ref{LemCvgSumY4} together with the invertibility of $S_{n-1}$ (assuming a suitable choice of $S$ in \eqref{Sn}) and $\Delta_{p}$ (Lemma \ref{LemInvD}), we have
\begin{equation*}
\limn n\, S_{n-1}^{-1} = \Delta_{p}^{-1} \cvgas \hsp \text{and} \hsp \VVert R_{n} \VVert = o(\sqrt{n}) \cvgas
\end{equation*}
provided that $\dE[V_1^4] = \tau^4 < \infty$. It follows that
\begin{equation}
\label{EstTTlcCvgRes}
\limn \sqrt{n}\, \vec(S_{n-1}^{-1}\, R_{n}) = 0 \cvgas
\end{equation}
Moreover, it is not hard to see that
\begin{equation}
\label{EstTTlcCvgMart1}
\vec(S_{n-1}^{-1}\, M_{n}) = (I_{q} \otimes S_{n-1}^{-1})\, \vec(M_{n})
\end{equation}
where $\vec(M_{n})$ is a vector $(\cF_{n})$--martingale of order $p\, q$, and that
\begin{equation}
\label{EstTTlcCvgMart2}
\limn n\, (I_{q} \otimes S_{n-1}^{-1}) = I_{q} \otimes \Delta_{p}^{-1} \cvgas
\end{equation}
The predictable quadratic variation of $\vec(M_{n})$ is given, for all $n \geq 1$, by
\begin{equation*}
\langle \vec(M) \rangle_{n} = \begin{pmatrix}
\langle M_1, M_1 \rangle_{n} & \langle M_1, M_2 \rangle_{n} & \hdots & \langle M_1, M_{q} \rangle_{n} \\
\langle M_2, M_1 \rangle_{n} & \langle M_2, M_2 \rangle_{n} & & \vdots \\
\vdots & & \ddots & \vdots \\
\langle M_{q}, M_1 \rangle_{n} & \hdots & \hdots & \langle M_{q}, M_{q} \rangle_{n}
\end{pmatrix}
\end{equation*}
where, for all $h, k \in \{1, \hdots, q \}$,
\begin{equation*}
\langle M_{h}, M_{k} \rangle_{n} = \sigma^2 \sum_{t=1}^{n} \Phi_{t-h}^{p}\, \Phi_{t-k}^{p\,\, \prime}.
\end{equation*}
We are now able to establish the asymptotic behavior of $\langle \vec(M) \rangle_{n}$ as a function of the asymptotic covariances $\ell_0, \hdots, \ell_{p+q}$ defined in \eqref{StatL0} and \eqref{StatLh}. Indeed,
\begin{equation*}
\limn \frac{\langle M_{h}, M_{k} \rangle_{n}}{n} = \sigma^2\, \Gamma_{h-k} \cvgas
\end{equation*}
and $\Gamma_{h-k}$ is given in \eqref{Gammahk}. We deduce that
\begin{equation}
\label{EstTTlcCvgMart3}
\limn \frac{\langle \vec(M) \rangle_{n}}{n} = \sigma^2\, \Gamma_{\!pq} \cvgas
\end{equation}
where $\Gamma_{\!pq}$ is precisely \eqref{Gamma}. In addition, the Lindeberg's condition is satisfied for $\vec(M_{n})$. As a matter of fact, if we denote by
\begin{equation}
\label{PhiMartVec}
\varphi_{t-1}^{p\, q} = \begin{pmatrix}
\Phi_{t-1}^{p\,\, \prime} & \Phi_{t-2}^{p\,\, \prime} & \hdots & \Phi_{t-q}^{p\,\, \prime}
\end{pmatrix}^{\prime}
\end{equation}
which is a vector of order $p\, q$, then
\begin{equation*}
\vec(M_{n}) = \sum_{t=1}^{n} \varphi_{t-1}^{p\, q}\, V_{t}.
\end{equation*}
This formulation will be easier to handle in what follows. For all $\veps > 0$,
\begin{eqnarray}
\label{Lindeberg}
\frac{1}{n} \sum_{t=1}^{n} \dE \left[ \Vert \vec(\Delta M_{t}) \Vert^2 ~ \dI_{\left\{ \Vert \vec(\Delta M_{t}) \Vert\, \geq\, \veps \sqrt{n} \right\} }\, \vert\, \cF_{t-1} \right] & \leq & \frac{\tau^4}{\veps^2\, n^2} \sum_{t=1}^{n} \Vert \varphi_{t-1}^{p\, q} \Vert^4\nonumber \\
 & = & O(n^{-1}) \cvgas
\end{eqnarray}
\textit{via} Lemma \ref{LemCvgSumY4}. Consequently, we infer from the central limit theorem for vector martingales (see Corollary 2.1.10 of \cite{Duflo97}) that we have the asymptotic normality
\begin{equation}
\label{EstTTlcCvgMart4}
\frac{\vec(M_{n})}{\sqrt{n}} \liml \cN\big( 0, \sigma^2\, \Gamma_{\!pq} \big).
\end{equation}
Whence we deduce from \eqref{EstTTlcCvgMart1}, \eqref{EstTTlcCvgMart2}, \eqref{EstTTlcCvgMart4} and Slutsky's lemma that
\begin{equation}
\label{EstTTlcCvgMart5}
\sqrt{n}\, K^{-1}\, \vec(S_{n-1}^{-1}\, M_{n}) \liml \cN\big( 0,  \sigma^2\, K^{-1}\, (I_{q} \otimes \Delta_{p}^{-1})\, \Gamma_{\!pq}\, (I_{q} \otimes \Delta_{p}^{-1})\, K^{\, \prime\, -1} \big).
\end{equation}
Together with \eqref{EstTVecEqTlc} and \eqref{EstTTlcCvgRes}, this achieves the proof of Proposition \ref{PropTlcTn} for $p \geq q$. The proof for $p < q$ is similar.
\hfill \qedsymbol

\bigskip

\noindent \textbf{Proof of Proposition \ref{PropRatTn}.} Let us get back to \eqref{EstTDecomp} which leads together with \eqref{EstTTlcCvgMart1} to
\begin{equation}
\label{EstTDecomp2}
\vec(T_{n}) - \vec(T^{*}) = K^{-1} (I_{q} \otimes S_{n-1}^{-1})\, \vec(M_{n}) + K^{-1}\, \vec(S_{n-1}^{-1}\, R_{n}).
\end{equation}
We will not develop entirely the proof of this proposition, since it follows exactly the same lines as the proof of Theorem 2.3 in \cite{Proia13}. We first establish, using Lemma \ref{LemCvgSumY4} and the same notations as above, that
\begin{equation}
\label{EstTRatCond}
\sum_{n=1}^{\infty} \frac{\Vert \varphi_{n-1}^{p\, q} \Vert^4}{n^2} < \infty \cvgas
\end{equation}
From \eqref{EstTTlcCvgMart3}, \eqref{EstTRatCond} and using Theorem 2.1 of \cite{ChaabaneMaaouia00}, we infer that the vector $(\cF_{n})$--martingale $\vec(M_{n})$ satisfies the quadratic strong law described by
\begin{equation}
\label{EstTRatLfq}
\limn \frac{1}{\log n} \sum_{t=1}^{n} \frac{\vec(M_{t})\, \vec(M_{t})^{\prime}}{t^2} = \sigma^2\, \Gamma_{\!pq} \cvgas
\end{equation}
where $\Gamma_{\!pq}$ is given in \eqref{Gamma}. From \eqref{EstTDecomp2}, we obtain that
\begin{equation}
\label{EstTRatDecompLfq}
\big( V_{T_{n}} - V_{T^{*}} \big) \big( V_{T_{n}} - V_{T^{*}} \big)^{\prime} = K^{-1}\, (I_{q} \otimes S_{n-1}^{-1})\, V_{M_{n}}\, V_{M_{n}}^{\, \prime} (I_{q} \otimes S_{n-1}^{-1})\, K^{\, \prime\, -1} + \zeta_{n}
\end{equation}
where $V_{T_{n}} = \vec(T_{n})$, $V_{T^{*}} = \vec(T^{*})$, $V_{M_{n}} = \vec(M_{n})$ and where the remainder term is given, for all $n \geq 1$, by
\begin{equation*}
\zeta_{n} = K^{-1}\, (I_{q} \otimes S_{n-1}^{-1})\,\left[ V_{M_{n}}\, V_{R_{n}}^{\, \prime} + V_{R_{n}}\, V_{M_{n}}^{\, \prime} + V_{R_{n}}\, V_{R_{n}}^{\, \prime} \right]\, (I_{q} \otimes S_{n-1}^{-1})\, K^{\, \prime\, -1}
\end{equation*}
with $V_{R_{n}} = \vec(R_{n})$. We also know from Lemma \ref{LemCvgSumY4} that, under our hypotheses, $\Vert \vec(R_{n}) \Vert = o(\sqrt{n})$ a.s. and especially
\begin{equation*}
\Vert \vec(R_{n}) \Vert = o(\Vert \vec(M_{n}) \Vert) \cvgas
\end{equation*}
The latter remark together with \eqref{EstTRatLfq} directly shows that
\begin{equation}
\label{EstTRatLfqRes}
\limn \frac{1}{\log n} \sum_{t=1}^{n} \zeta_{t} = 0 \cvgas
\end{equation}
The combination of \eqref{EstTTlcCvgMart2}, \eqref{EstTRatLfq} \eqref{EstTRatDecompLfq} and \eqref{EstTRatLfqRes} concludes the first part of the proof. The law of iterated is much more easy to handle. From Lemma C.2 in \cite{Bercu98}, for every $v \in \dR^{p\, q}$,
\begin{equation}
\label{EstTLliMart}
\limsup_{n\, \rightarrow\, \infty} \frac{v^{\,\prime}\, \vec(M_{n})}{\sqrt{2\, n \log \log n}} = - \liminf_{n\, \rightarrow\, \infty} \frac{v^{\,\prime}\, \vec(M_{n})}{\sqrt{2\, n \log \log n}} = \sigma \sqrt{v^{\,\prime}\, \Gamma_{\!pq}\, v} \cvgas
\end{equation}
Hence,
\begin{equation}
\label{EstTLliRes}
\Vert \vec(S_{n-1}^{-1}\, R_{n}) \Vert = o\!\left( \sqrt{\frac{\log \log n}{n}} \right) \cvgas
\end{equation}
We obtain the matrix formulation
\begin{equation*}
\limsup_{n\, \rightarrow\, \infty} \left( \frac{n}{2 \log \log n} \right)\, \big( V_{T_{n}} - V_{T^{*}} \big) \big( V_{T_{n}} - V_{T^{*}} \big)^{\prime} = \Sigma_{T} \cvgas
\end{equation*}
where $\Sigma_{T}$ is given in \eqref{TCov}, which achieves the proof for $p \geq q$. Once again, the proof follows exactly the same lines for $p < q$.
\hfill \qedsymbol

\bigskip

\noindent \textbf{Proof of Theorem \ref{ThmCvgR}.} This proof is tedious but quite straightforward. Indeed, for all $1 \leq t \leq n$, consider
\begin{equation}
\label{Ppq}
P_{t}^{p\, q} = \begin{pmatrix}
\Phi_{t}^{q} & \Phi_{t-1}^{q} & \hdots & \Phi_{t-p+1}^{q}
\end{pmatrix}
\end{equation}
with initial values set to zero, and note that
\begin{equation*}
\wh{J}_{n} = \sum_{t=0}^{n} \Phi_{t}^{q}\, \Phi_{t}^{q\,\, \prime} - \sum_{t=1}^{n} P_{t-1}^{p\, q}\, \wht_{n}\, \Phi_{t}^{q\,\, \prime} - \sum_{t=1}^{n} \Phi_{t}^{q}\, \wht_{n}^{\: \prime}\, P_{t-1}^{p\, q\,\, \prime} + \sum_{t=1}^{n} P_{t-1}^{p\, q}\, \wht_{n}\, \wht_{n}^{\: \prime}\, P_{t-1}^{p\, q\,\, \prime} + J
\end{equation*}
and that
\begin{equation*}
\sum_{t=1}^{n} \wh{\Psi}_{t-1}^{q}\, \wh{Z}_{t} = \sum_{t=1}^{n} \Phi_{t-1}^{q}\, Y_{t} - \sum_{t=1}^{n} P_{t-2}^{p\, q}\, \wht_{n}\, Y_{t} - \sum_{t=1}^{n} \Phi_{t-1}^{q}\, \wht_{n}^{\: \prime}\, \Phi_{t-1}^{p} + \sum_{t=1}^{n} P_{t-2}^{p\, q}\, \wht_{n}\, \wht_{n}^{\: \prime}\, \Phi_{t-1}^{p}
\end{equation*}
where $\Phi_{t}^{q}$ is given in \eqref{PhiPsi}, $\wh{Z}_{t}$ in \eqref{EstRes}, $\wh{J}_{n}$ in \eqref{Jn} and $\wh{\Psi}_{t}^{q}$ in \eqref{PsiHat}. The end of the proof is achieved by considering the definition of the matrices $\Delta_{h}$ in \eqref{Delta}, $\Gamma_{h-k}$ in \eqref{Gammahk}, and repeatedly making use of Lemma \ref{LemCvgSumY2} to establish the limiting values.
\hfill \qedsymbol

\bigskip

\noindent \textbf{Proof of Theorem \ref{ThmTlcR}.} Under $\cH_0 : ``\rho_1 = \hdots = \rho_{q} = 0"$, all calculations are simplified. In particular, it is easy to establish that, for all $n \geq 1$,
\begin{equation*}
\wh{\theta}_{n} - \theta = S_{n-1}^{-1}\, M_{n} + r_{n}
\end{equation*}
where, as soon as $\dE[V_1^4] = \tau^4 < \infty$, $r_{n} = o(\sqrt{n})$ a.s. and
\begin{equation}
\label{EstRMn}
M_{n} = \sum_{t=1}^{n} Y_{t-1} V_{t}
\end{equation}
is a $(\cF_{n}$)--martingale. The first step of our reasoning is to prove that, under $\cH_0$,
\begin{equation}
\label{CvgDenRnH0}
\limn \frac{\wh{J}_{n}}{n} = \sigma^2 I_{q} \cvgas
\end{equation}
where $\wh{J}_{n}$ is given in \eqref{Jn}. For all $1 \leq h \leq q-1$, we have \textit{via} \eqref{EstRes},
\begin{eqnarray*}
\sum_{t=1}^{n} \wh{Z}_{t-h} \wh{Z}_{t} & = & \sum_{t=1}^{n} \big( Y_{t-h} - \wh{\theta}_{n}^{\: \prime}\, \Phi_{t-h-1}^{p} \big) \big( Y_{t} - \wh{\theta}_{n}^{\: \prime}\, \Phi_{t-1}^{p} \big) \\
 & = & \sum_{t=1}^{n} Y_{t-h} Y_{t} - \wh{\theta}_{n}^{\: \prime} \sum_{t=1}^{n} \Phi_{t-h-1}^{p} Y_{t} - \wh{\theta}_{n}^{\: \prime} \sum_{t=1}^{n} \Phi_{t-1}^{p} Y_{t-h} + \wh{\theta}_{n}^{\: \prime} \sum_{t=1}^{n} \Phi_{t-h-1}^{p} \Phi_{t-1}^{p\: \prime}\, \wh{\theta}_{n} 
\end{eqnarray*}
with initial values set to zero, as usual. Then, the simplifications established in \eqref{CovH0}, the related notations together with Lemma \ref{LemCvgSumY2} imply that
\begin{eqnarray*}
\limn \frac{1}{n} \sum_{t=1}^{n} \wh{Z}_{t-h} \wh{Z}_{t} & = & \ell_{h} - \theta^{\, \prime} \Lambda_{-h+1}^{p} - \theta^{\, \prime} \Lambda_{h+1}^{p} + \theta^{\, \prime}\, \Gamma_{h}\, \theta \\
 & = & \ell_{h} - \theta^{\, \prime} \Lambda_{-h+1}^{p} ~ = ~ 0 \cvgas
\end{eqnarray*}
since, from the Yule-Walker equations, the asymptotic covariance of order $h$ satisfies $\ell_{h} = \theta_1\, \ell_{h-1} + \hdots + \theta_{p}\, \ell_{h-p}$. For $h=0$, it is not hard to establish the convergence to $\sigma^2$ using exactly the same lines, which proves \eqref{CvgDenRnH0}. Besides, a direct calculation under $\cH_0$ leads to
\begin{equation}
\label{EstRTlcDecomp}
\sum_{t=1}^{n} \wh{\Psi}_{t-1}^{q}\, \wh{Z}_{t} = A_{n} + R_{1,\, n} + R_{2,\, n} + R_{n}
\end{equation}
with
\begin{eqnarray*}
A_{n} & = & \sum_{t=1}^{n} \Phi_{t-1}^{q}\, Y_{t} - \sum_{t=1}^{n} P_{t-2}^{p\, q}\, \theta\, Y_{t} - \sum_{t=1}^{n} \Phi_{t-1}^{q}\, \theta^{\, \prime}\, \Phi_{t-1}^{p} + \sum_{t=1}^{n} P_{t-2}^{p\, q}\, \theta\, \theta^{\, \prime}\, \Phi_{t-1}^{p}, \\
R_{1,\, n} & = & \sum_{t=1}^{n} \big( P_{t-2}^{p\, q}\, \theta - \Phi_{t-1}^{q} \big) \Phi_{t-1}^{p\: \prime}\, \big( \wh{\theta}_{n} - \theta \big) - \sum_{t=1}^{n} P_{t-2}^{p\, q}\, \big( \wh{\theta}_{n} - \theta \big) \big( Y_{t} - \theta^{\, \prime}\, \Phi_{t-1}^{p} \big), \\
R_{2,\, n} & = & \sum_{t=1}^{n} P_{t-2}^{p\, q}\, \big( \wh{\theta}_{n} - \theta \big) \big( \wh{\theta}_{n} - \theta \big)^{\prime}\, \Phi_{t-1}^{p},
\end{eqnarray*}
in which $P_{t}^{p\, q}$ is given in \eqref{Ppq} and $R_{n}$ is a residual made of isolated terms satisfying, by virtue of Lemma \ref{LemCvgSumY4}, $\Vert R_{n} \Vert = o(\sqrt{n})$ a.s. On the one hand, from Lemma \ref{LemCvgSumY2} together with Theorem \ref{ThmRatT} (and particularly \eqref{EstTRatNorm}), we obtain the upper bound
\begin{equation*}
\Vert R_{2,\, n} \Vert ~ \leq ~ \big\Vert \wh{\theta}_{n} - \theta \big\Vert^2\, O(n) ~ = ~ O(\log \log n) \cvgas
\end{equation*}
which clearly implies that
\begin{equation}
\label{EstRTlcR2}
\limn \frac{R_{2,\, n}}{\sqrt{n}} = 0 \cvgas
\end{equation}
On the other hand, for all $1 \leq t \leq n$, $Y_{t} = \theta^{\, \prime}\, \Phi_{t-1}^{p} + V_{t}$. It follows that
\begin{equation}
\label{EstRDecompA}
A_{n} = \sum_{t=1}^{n} \big( \Phi_{t-1}^{q} - P_{t-2}^{p\, q}\, \theta \big) V_{t} = \sum_{t=1}^{n} \varphi_{t-1}^{q} V_{t}
\end{equation}
where
\begin{equation*}
\varphi_{t}^{\, q} = \begin{pmatrix}
V_{t} & V_{t-1} & \hdots & V_{t-q+1}
\end{pmatrix}^{\prime}.
\end{equation*}
Using the same methodology, we find that
\begin{equation*}
\Big\Vert \sum_{t=1}^{n} P_{t-2}^{p\, q}\, \big( \wh{\theta}_{n} - \theta \big) V_{t} \Big\Vert \leq \big\Vert \wh{\theta}_{n} - \theta \big\Vert\, \Big\Vert \sum_{t=1}^{n} P_{t-2}^{p\, q}\, V_{t} \Big\Vert = O_{\! p}\big( \sqrt{\log \log n} \big)
\end{equation*}
from Theorem \ref{ThmRatT} and using the central limit theorem for vector martingales (see Corollary 2.1.10 of \cite{Duflo97}) which immediately gives the $O_{\! p}(\sqrt{n})$ rate for the right-hand side of the inequality. Consequently, \textit{via} \eqref{EstRMn},
\begin{equation}
\label{EstRTlcR1}
\frac{R_{1,\, n}}{\sqrt{n}} = -\frac{1}{{\sqrt{n}}}\, \sum_{t=1}^{n} \varphi_{t-1}^{\, q} \Phi_{t-1}^{p\: \prime}\, S_{n-1}^{-1}\, \sum_{t=1}^{n} \Phi_{t-1}^{p} V_{t} \: + \: o_{p}(1).
\end{equation}
Now, consider the vector $(\cF_{n})$--martingale of order $p+q$ given, for all $n \geq 1$, as
\begin{equation}
\label{Nn}
N_{n} = \sum_{t=1}^{n} H_{t-1}^{p+q}\, V_{t} \hsp \text{with} \hsp H_{t} = \begin{pmatrix}
\varphi_{t}^{\,q\: \prime} & \Phi_{t}^{p\: \prime}
\end{pmatrix}^{\prime}.
\end{equation}
From \eqref{EstRTlcDecomp}--\eqref{Nn}, we can deduce that
\begin{equation}
\label{EstRTlcDecomp2}
\frac{1}{\sqrt{n}}\, \sum_{t=1}^{n} \wh{\Psi}_{t-1}^{q}\, \wh{Z}_{t} = \frac{1}{\sqrt{n}}\, U_{n-1}\, N_{n} \: + \: o_{p}(1)
\end{equation}
where, for all $n \geq 1$,
\begin{equation*}
\Upsilon_{n} = \sum_{t=1}^{n} \varphi_{t}^{\, q}\, \Phi_{t}^{p\: \prime} \hsp \text{and} \hsp U_{n} = \begin{pmatrix}
I_{q} & -\Upsilon_{\!n}\, S_{n}^{-1}
\end{pmatrix}.
\end{equation*}
The predictable quadratic variation of $N_{n}$ is given, for all $n \geq 1$, by
\begin{equation}
\label{MartCrochNn}
\langle N \rangle_{n} = \sum_{t=1}^{n} \dE[ (\Delta N_{t})(\Delta N_{t})^{\prime}\, \vert\, \cF_{t-1} ] = \sigma^2 \sum_{t=1}^{n} H_{t-1}^{p+q}\, H_{t-1}^{p+q\: \prime}
\end{equation}
and satisfies, after some additional calculations, 
\begin{equation*}
\limn \frac{\langle N \rangle_{n}}{n} = \begin{pmatrix}
\sigma^4 I_{q} & \sigma^2\, \Upsilon_{\!pq} \\
\sigma^2\, \Upsilon_{\!pq}^{\, \prime} & \sigma^2\, \Delta_{p}
\end{pmatrix} = \Gamma_{N} \cvgas
\end{equation*}
where $\Upsilon_{\!pq}$ is the almost sure limit of $\Upsilon_{n}/n$, explicitly given in \eqref{Ups1} (or \eqref{Ups2}). In addition, $N_{n}$ satisfies the Lindeberg's condition (following the same lines as what we did in \eqref{Lindeberg}). Hence, from \eqref{EstRTlcDecomp2}, Slutsky's lemma and once again the central limit theorem for vector martingales,
\begin{equation*}
\frac{1}{\sqrt{n}}\, \sum_{t=1}^{n} \wh{\Psi}_{t-1}^{q}\, \wh{Z}_{t} \liml \cN\big(0, \Gamma_{Z} \big)
\end{equation*}
where
\begin{eqnarray*}
\Gamma_{Z} & = & \begin{pmatrix}
I_{q} & -\Upsilon_{\!pq}\,\Delta_{p}^{-1}
\end{pmatrix} \Gamma_{N} \begin{pmatrix}
I_{q} & -\Upsilon_{\!pq}\,\Delta_{p}^{-1}
\end{pmatrix}^{\prime} \\
 & = & \begin{pmatrix}
\sigma^4 I_{q} - \sigma^2\, \Upsilon_{\!pq}\, \Delta_{p}^{-1}\, \Upsilon_{\!pq}^{\: \prime} & 0
\end{pmatrix} \begin{pmatrix}
I_{q} & -\Upsilon_{\!pq}\,\Delta_{p}^{-1}
\end{pmatrix}^{\prime} = \sigma^4 I_{q} - \sigma^2\, \Upsilon_{\!pq}\, \Delta_{p}^{-1}\, \Upsilon_{\!pq}^{\: \prime}.
\end{eqnarray*}
The definition of $\whr_{n}$ in \eqref{EstR} together with \eqref{CvgDenRnH0} and Slutsky's lemma imply that $\Sigma_{\rho}^{\, 0} = \sigma^{-4} \Gamma_{Z}$, which achieves the proof.
\hfill \qedsymbol

\bigskip

\nocite{*}

\bibliographystyle{acm}
\bibliography{ARARModels}

\medskip

\end{document}